\documentclass[]{amsart}
\usepackage[utf8]{inputenc}
\usepackage[T1]{fontenc}
\usepackage{lmodern}
\usepackage{microtype}
\usepackage{csquotes}
\usepackage{mathtools}
\usepackage[all]{xy}
\usepackage[american]{babel}
\usepackage[
url=false,
backend=bibtex
]{biblatex}
\usepackage[
colorlinks,
]{hyperref}
\usepackage[
]{cleveref}
\bibliography{bibfile}
\theoremstyle{plain}
\newtheorem{theo}{Theorem}[section]
\newtheorem{lemm}[theo]{Lemma}
\newtheorem{coro}[theo]{Corollary}
\newtheorem{prop}[theo]{Proposition}
\theoremstyle{definition}
\newtheorem{defi}[theo]{Definition}
\theoremstyle{remark}
\newtheorem{rema}[theo]{Remark}
\newtheorem{exam}[theo]{Example}
\numberwithin{equation}{theo}

\newcommand{\N}{\mathbb{N}}
\newcommand{\Z}{\mathbb{Z}}
\newcommand{\R}{\mathbb{R}}
\newcommand{\id}{\mathrm{id}}
\newcommand{\cL}{\mathcal{L}}
\DeclareMathOperator{\codim}{codim}
\DeclarePairedDelimiter\paren{(}{)}
\DeclarePairedDelimiter\abs{\lvert}{\rvert}
\DeclarePairedDelimiter\norm{\lVert}{\rVert}
\renewcommand{\bar}{\overline}
\renewcommand{\tilde}{\widetilde}
\renewcommand{\mid}{:}
\renewcommand{\vee}{*}

\begin{document}
\title{Extendability of parallel sections in vector bundles}
\author{Tim Kirschner}
\address{Fakultät für Mathematik\\ Universität Duisburg-Essen}
\email{tim.kirschner@uni-due.de}
\urladdr{http://timkirschner.tumblr.com}
\date\today
\subjclass[2010]{Primary 53C05; Secondary 53B05, 14J60, 53C29}
\begin{abstract}
I address the following question: Given a differentiable manifold $M$, what are the open subsets $U$ of $M$ such that, for all vector bundles $E$ over $M$ and all linear connections $\nabla$ on $E$, any $\nabla$-parallel section in $E$ defined on $U$ extends to a $\nabla$-parallel section in $E$ defined on $M$?

For simply connected manifolds $M$ (among others) I describe the entirety of all such sets $U$ which are, in addition, the complement of a $C^1$ submanifold, boundary allowed, of $M$. This delivers a partial positive answer to a problem posed by Antonio J. Di Scala and Gianni Manno \cite{discala}. Furthermore, in case $M$ is an open submanifold of $\mathbb R^n$, $n\ge 2$, I prove that the complement of $U$ in $M$, not required to be a submanifold now, can have arbitrarily large $n$-dimensional Lebesgue measure.
\end{abstract}
\maketitle
\tableofcontents

\section{Introduction}
\label{s_intro}

In their recent preprint Antonio J. Di Scala and Gianni Manno raise the following question \cite[Problem 1]{discala}: given a vector bundle $E$ over a simply connected manifold $M$, a connection $\nabla$ on $E$, and a $\nabla$-parallel section $\sigma$ in $E$ defined on an open, dense, connected subset $U \subset M$, does there exist a $\nabla$-parallel section $\tilde\sigma$ defined on $M$ such that $\tilde\sigma$ extends $\sigma$---that is, such that $\tilde\sigma|_U = \sigma$? Di Scala and Manno explain, among others, how to apply this question to the extension of Killing vector fields. Please consult their introduction for details as well as further applications.

For the note at hand, I would like to widen the scope of Di Scala's and Manno's question slightly suggesting an alternative problem: for a given manifold $M$ (simply connected or not), describe/characterize the set of all open subsets $U \subset M$ such that, for all vector bundles $E$ over $M$, all connections $\nabla$ on $E$, and all $\nabla$-parallel sections $\sigma$ in $E$ defined on $U$, there exists a $\nabla$-parallel extension $\tilde\sigma$ as above. As a matter of fact, I will try and characterize the universe of closed subsets $F \subset M$ whose complement $U = M \setminus F$ has the aforementioned property. These closed subsets $F \subset M$ will be called \emph{negligible} in $M$ (see \cref{d_negligible}).

My results come in two groups. For one thing, in \cref{s_basics}, I derive \emph{necessary conditions} for a set $F$ to be negligible in $M$. Specifically, I prove that when $F$ is negligible in $M$, and $M$ is connected, then the complement $M \setminus F$ is necessarily connected too (\cref{p_connected}). When $M$ is of dimension $2$ or higher, then, moreover, $F$ needs to be nowhere dense in $M$ (\cref{c_nowheredense}). Observe that Di Scala and Manno have already pointed these two conditions out as relevant---without, however, proving their necessity \cite{discala}.

For another, in \cref{s_hyper}, I derive \emph{sufficient conditions} for a set $F$ to be negligible in $M$. This is probably the more interesting part (as compared to \cref{s_basics}) since here I prove that parallel extensions of parallel sections do in fact exist. The most striking result of \cref{s_hyper} is \cref{c_submfds} which asserts in particular (compare \cref{r_dissected}) that when $M$ is a simply connected (second-countable, Hausdorff) manifold and $F \subset M$ is a closed $C^1$ submanifold with boundary such that $M \setminus F$ is dense and connected in $M$, then $F$ is negligible in $M$. Hence, \cref{c_submfds} yields a partial positive answer for the question of Di Scala and Manno.

As a sideline in \cref{s_hyper}, I show that the Lebesgue measure of a set $F$ is quite unrelated to the negligibility of $F$. Precisely, I prove the existence of negligible subsets of $\R^n$ of arbitrarily large, and even infinite, measure (\cref{c_bigmeasure}). Besides, and constrasting, I show that the fact that $F$ has Lebesgue measure $0$ inside $\R^n$, $n\ge 2$, does not imply that $F$ is negligible in $\R^n$ for all connections of class $C^0$, even when $\R^n\setminus F$ is connected (\cref{c_counterex-c0}). The latter observation tells us that the smoothness of the connection is essential in Di Scala's and Manno's question.

\Cref{s_manifolds,s_connections} contain preliminary definitions, conventions, and remarks that I employ in the course of \cref{s_basics,s_hyper}.
\medskip

\emph{Acknowledgements}: I would like to thank Antonio J. Di Scala for introducing me to the circle of problems at hand in the first place---not least by giving a very beautiful talk on the subject in Bayreuth last year. I would like to thank Gerhard Rein for letting me present my research (albeit slightly off-topic) in his Oberseminar. Finally, my very special thanks goes to Thomas Kriecherbauer for a) suggesting to look at “fat” Cantor sets in connection with \cref{l_hyper} and b) encouraging me to write down \cref{t_submfds} in arbitrary dimension (where I was treating merely the case $\dim M = 2$ at the time).

\section{Manifolds and submanifolds with boundary}
\label{s_manifolds}

By a \emph{manifold} I mean a locally finite-dimensional (i.e., locally modeled on some $\R^n$, $n \in \N$) differentiable manifold of class $C^k$, $1 \leq k \leq \infty$, without boundary; I make no topological assumptions whatsoever \cite[cf.][23]{lang}. The extended (i.e., allowing $\infty$) natural number $k$ will be fixed throughout. The sheaf of real-valued functions of class $C^m$ on $M$, $0 \leq m \leq k$, will be denoted by $C^m_M$, or by $C^m$ when $M$ is clear from the context.

Let us recall, mainly for the sake of \cref{t_submfds}, some terminology concerning manifolds and submanifolds with boundary.\footnote{The manifolds with boundary that we consider will always arise as submanifolds.}

\begin{defi}
\label{d_submfd}
Let $M$ be a manifold, $F \subset M$ a subset, $0 \leq m \leq k$ a natural number or $\infty$.

Let $p \in F$. Then $F$ is a \emph{$C^m$ submanifold with boundary} of $M$ at $p$ when there exist $d,c \in \N$, an open neighborhood $U$ of $p$ in $M$, an open subset $V$ of $\R^d \times \R^c$, and an isomorphism $\phi \colon U \to V$ of class $C^m$ such that
\[
\phi(F \cap U) = (H_d \times \{(0,\dots,0)\}) \cap V,
\]
where
\begin{equation}
\label{e_halfspace}
H_d := \begin{cases}\{(x_1,\dots,x_d) \in \R^d \mid 0 \leq x_d\} & \text{when } d>0, \\ \R^0 & \text{when } d=0. \end{cases}
\end{equation}
In that case we write
\[
\codim_p(F,M) = c.
\]

We say that $F$ is a \emph{$C^m$ submanifold with boundary} of $M$ if $F$ is a $C^m$ submanifold with boundary of $M$ at $p$ for all $p \in F$. If this is the case, we set
\[
\codim(F,M) := \inf \{\codim_p(F,M) \mid p \in F\},
\]
where the infimum of the empty set is taken to be $\infty$.
\end{defi}

In \cref{s_hyper} we are interested in nowhere dense, closed subsets $F$ of manifolds $M$. In case $F$ is a submanifold with boundary of $M$, nowhere density can be characterized in terms of the codimension of $F$ inside $M$.

\begin{rema}[Density and codimension]
\label{r_nowdense-codim}
Let $M$ be a manifold, $F$ a closed $C^0$ submanifold with boundary of $M$. Then the following are equivalent:
\begin{enumerate}
\item \label{i_nowdense} $F$ is nowhere dense in $M$.
\item \label{i_codim} $1 \leq \codim(F,M)$.
\end{enumerate}
Assume \cref{i_nowdense}. Let $p \in F$. Assume $\codim_p(F,M) = 0$. Then there exists an open neighborhood $U$ of $p$ in $M$, a number $n \in \N$, an open subset $V \subset \R^n$, and a homeomorphism $\phi \colon U \to V$ such that $\phi(F \cap U) = H_n \cap V$. If $n=0$, then $U = \{p\} \subset F$. Thus $p$ is an interior point of $F$ in $M$, contradicting \cref{i_nowdense}. If $n\ge1$, then there exists a number $\epsilon>0$ such that
\[
y := \phi(p) + (0,\dots,0,\epsilon) \in \{x \in \R^n \mid 0 < x_n\} \cap V.
\]
Thus the preimage of the point $y$ is an interior point of $F$ in $M$---contradiction. So, $1 \leq \codim_p(F,M)$. As $p \in F$ was arbitrary, we have \cref{i_codim}.

Assume \cref{i_codim}, and let $p \in F$. Then there exists an open neighborhood $U$ of $p$ in $M$, numbers $d,c \in \N$, an open subset $V$ of $\R^d \times \R^c$, and a homeomorphism $\phi \colon U \to V$ such that $\phi(F \cap U) \subset (\R^d \times \{(0,\dots,0)\}) \cap V$. If $p$ were an interior point of $F$ in $M$, the point $\phi(p)$ would be an interior point of $\R^d \times \{(0,\dots,0)\}$ in $\R^d \times \R^c$. Yet as $1 \leq \codim(F,M) \leq \codim_p(F,M) = c$, this is absurd.
\end{rema}

A manifold (or submanifold) with boundary has interior points and boundary points. The precise definition of these notions is given below. A word of caution: When $X$ is a topological space and $A$ a subset of $X$, we have the usual, topological notion of an interior point of $A$ in $X$---for instance, as used in \cref{r_nowdense-codim}. Furthermore, we have the usual notion of a boundary point of $A$ in $X$. These notions have to be distinguished carefully from the following, even though the names are the same.

\begin{defi}
\label{d_intpoint}
Let $F$ be an arbitrary topological space, $p \in F$.
We say $p$ is an \emph{interior point} (in the manifold sense) of $F$ when there exist $d \in \N$, an open neighborhood $U'$ of $p$ in $F$, and a homeomorphism $\phi' \colon U' \to \R^d$. The set of interior points of $F$ is called the \emph{interior} of $F$.

We say $p$ is a \emph{boundary point} (in the manifold sense) of $F$ when there exist $d \in \N$, an open neighborhood $U'$ of $p$ in $F$, and a homeomorphism $\phi' \colon U' \to H_d$, where $H_d$ is given by \cref{e_halfspace}, such that $\phi'(p)$ lies in the topological boundary of $H_d$ in $\R^d$. The set of boundary points of $F$ will be denoted by $\partial F$.
\end{defi}

\section{Linear connections in vector bundles}
\label{s_connections}

By a \emph{vector bundle} over a manifold I mean a locally finite-rank real vector bundle of class $C^l$, where $0 \leq l \leq k$. When $E$ is a vector bundle over $M$ and $0 \leq m \leq l$, I write $C^m(E)$ for the sheaf of sections of class $C^m$ in $E$; I write $C^m(U,E)$ as a synonym for $C^m(E)(U)$ when $U$ is open in $M$. Observe that an element $s \in C^m(U,E)$ is thus particularly a function $s \colon U \to E$.

For the purposes of this note, connections on vector bundles are understood in the sense of covariant derivatives. Thus, specifically, connections are always linear. Observe that even when working with smooth manifolds (i.e., when $l = k = \infty$), I do not require my connections to be smooth also, but merely continuous.

\begin{defi}
\label{d_connection}
A \emph{connection} on a vector bundle $E$ over $M$, where $1 \leq l$, is a morphism of abelian sheaves on $M$,
\[
\nabla \colon C^1(E) \to C^0(T^\vee(M) \otimes E),
\]
satisfying
\begin{equation}
\label{e_leibniz}
\nabla_U(fs) = df \otimes s + f\nabla_U(s)
\end{equation}
for all open subsets $U$ of $M$, all $f \in C^1(U)$, and all $s \in C^1(U,E)$.
If $\nabla$ is a connection on $E$, we put
\[
\Gamma^\nabla(E) := \ker(\nabla)
\]
so that $\Gamma^\nabla(E)$ becomes, in particular, an abelian subsheaf of $C^1(E)$. The sheaf $\Gamma^\nabla(E)$ is called the \emph{sheaf of $\nabla$-parallel sections} in $E$. We write $\Gamma^\nabla(U,E)$ as a synonym for $\Gamma^\nabla(E)(U)$.

We say that $\nabla$ is \emph{of class $C^m$}, $0 \leq m \leq \infty$, when $m+1 \leq l$ and $\nabla$ maps $C^l(E)$ into $C^m(T^\vee(M) \otimes E)$. Here, $\infty + 1 := \infty$. Note that a connection (without further specification) is hence the same thing as a connection of class $C^0$.

A \emph{vector bundle with connection (of class $C^m$)} over $M$ is a pair $(E,\nabla)$ such that $E$ is a vector bundle over $M$ and $\nabla$ is a connection (of class $C^m$) on $E$.
\end{defi}

\begin{rema}[Smooth connections]
\label{r_connection-class}
Typically connections are dealt with only when $E$ and (consequently) $M$ are of class $C^\infty$. A connection is then a morphism of abelian sheaves
\[
\nabla \colon C^\infty(E) \to C^\infty(T^\vee(M) \otimes E)
\]
satisfying Leibniz's rule---that is, \cref{e_leibniz}---for all $f \in C^\infty(U)$ and all $s \in C^\infty(U,E)$. Let me briefly reconcile this view with mine.

Evidently, when $\nabla$ is a connection on $E$ over $M$ in my sense, $\nabla$ induces, by restriction, a morphism of abelian sheaves
\[
\nabla' \colon C^l(E) \to C^0(T^\vee(M) \otimes E)
\]
satisfying \cref{e_leibniz} for all $f \in C^l(U)$, all $s \in C^l(U,E)$, and $U \subset M$ open.

Conversely, given such a morphism $\nabla'$, there exists a unique connection $\nabla$ on $E$ such that $\nabla$ restricted to $C^l(E) \subset C^1(E)$ equals $\nabla'$. As a matter of fact, on an open subset $U$ of $M$ over which $E$ is trivial (in the $C^l$ sense) the values of $\nabla_U$ are determined by ${\nabla'}_U$ and the Leibniz rule \cref{e_leibniz} for $f \in C^1(U)$ and $s \in C^l(U,E)$.

Moreover, under this correspondence, $\nabla$ is of class $C^m$, $0 \leq m \leq \infty$, if and only if $\nabla'$ has image lying in the subsheaf $C^m(T^\vee(M) \otimes E)$ of $C^0(T^\vee(M) \otimes E)$.
In particular, for $E$ and $M$ of class $C^\infty$, the usual $C^\infty$ connections on $E$ correspond naturally to my connections on $E$ which are of class $C^\infty$.
\end{rema}

\begin{rema}[The global section component]
\label{r_connection-global}
Let $0 \leq m \leq \infty$, $m+1 \leq l$, and
\[
\nabla_0 \colon C^l(M,E) \to C^m(M,T^\vee(M) \otimes E)
\]
be an additive map satisfying \cref{e_leibniz} for $\nabla_0$ in place of $\nabla_U$, for all $f \in C^l(M)$ and $s \in C^l(M,E)$.
Then, in case $M$ is Hausdorff, there exists a unique morphism of abelian sheaves
\[
\nabla \colon C^l(E) \to C^m(T^\vee(M) \otimes E)
\]
such that $\nabla_M = \nabla_0$ and \cref{e_leibniz} is satisfied for all $U \subset M$ open, all $f \in C^l(U)$, and all $s \in C^l(U,E)$.

The main reason is the following. When $M$ is Hausdorff, then for all $U \subset M$ open and all $p \in U$ there exists a function $f \in C^k(M)$ such that a), $f = 1$ on an open neighborhood $V$ of $p$ in $U$ and b), the support of $f$ is a (compact) subset of $U$. This fact implies that, for all $s \in C^l(U,E)$, there exists $\tilde s \in C^l(M,E)$ such that $\tilde s = s$ on $V$. Thus the calculation
\[
\nabla_U(s)|_V = \nabla_V(s|_V) = \nabla_V({\tilde s}|_V) = \nabla_M(\tilde s)|_V = \nabla_0(\tilde s)|_V
\]
shows the uniqueness of $\nabla$. The existence of $\nabla$ follows noting that when $\tilde s,\tilde s_1 \in C^l(M,E)$ agree on an open subset $V \cap V_1$ of $M$, then $\nabla_0(\tilde s)$ and $\nabla_0(\tilde s_1)$ agree on $V \cap V_1$, too. In fact, for $p \in V \cap V_1$, there exists a function $g \in C^k(M)$ such that a), $g = 1$ on a neighborhood $W$ of $p$ and b), the support of $g$ is contained in $V \cap V_1$.
In particular, $gt = 0$ on $M$, where $t = \tilde s_1 - \tilde s$, and thus, on $W$,
\[
0 = \nabla_0(gt) = dg \otimes t + g\nabla_0(t) = 0 \otimes t + 1\nabla_0(t) = \nabla_0(t) = \nabla_0(\tilde s_1) - \nabla_0(\tilde s)
\]
due to the additivity of $\nabla_0$ and \cref{e_leibniz} for $\nabla_0$ in place of $\nabla_U$.
\end{rema}

\begin{rema}[Connection forms]
\label{r_frame}
Let $M$ be a manifold, $E$ a vector bundle over $M$, $r \in \N$, and $(e_1,\dots,e_r)$ a frame (i.e., a global frame of class $C^l$, $1 \leq l$) for $E$. Then I contend that for all $r\times r$ matrices $(\omega^\alpha_\beta)$ with values in $C^0(M,T^\vee(M))$ there exists one, and only one, connection $\nabla$ on $E$ such that
\[
\nabla_M(e_\beta) = \omega^\alpha_\beta \otimes e_\alpha, \quad \forall \beta.
\]

For the uniqueness of $\nabla$ let $U$ be an open subset of $M$ and $\sigma \in C^1(U,E)$, and observe that there exists an $r$-tuple $(\sigma^\beta)$ of elements of $C^1(U)$ such that $\sigma = \sigma^\beta (e_\beta|_U)$. Therefore,
\begin{align*}
\nabla_U(\sigma) & = \nabla_U(\sigma^\beta (e_\beta|_U)) = d\sigma^\beta \otimes (e_\beta|_U) + \sigma^\beta \nabla_U(e_\beta|_U) \\
& = d\sigma^\beta \otimes (e_\beta|_U) + \sigma^\beta (\nabla_M(e_\beta)|_U) \\
& = d\sigma^\beta \otimes (e_\beta|_U) + \sigma^\beta (\omega^\alpha_\beta \otimes e_\alpha)|_U. 
\end{align*}

For the existence of $\nabla$ simply define $\nabla_U$, for $U \subset M$ open, by the latter identity---note that the tuple $(\sigma^\beta)$ is unique. Then the family $\nabla = (\nabla_U)$ is a morphism of sheaves on $M$,
\[
\nabla \colon C^1(E) \to C^0(T^\vee(M) \otimes E),
\]
mainly since, for $V \subset U$ open, $\sigma|_V = \sigma^\beta|_V (e_\beta|_V)$ and $(d\sigma^\beta)|_V = d(\sigma^\beta|_V)$. Furthermore, the morphism of sheaves $\nabla$ is additive; it satisfies the Leibniz rule as, for $f \in C^1(U)$, we have
\[
d(f\sigma^\beta) = \sigma^\beta df + f d\sigma^\beta, \quad \forall \beta.
\]
I omit the details.
\end{rema}

\begin{defi}
\label{d_trivialbundle}
When $M$ is a manifold and $r \in \N$, the projection $M \times \R^r \to M$ becomes a vector bundle $E$ of class $C^k$ over $M$ the obvious way. $E$ is called the \emph{trivial bundle} of rank $r$ over $M$. Let $(e_1,\dots,e_r)$ be the standard frame for $E$; that is,
\[
e_\alpha \colon M \to E, \quad e_\alpha(x) = (x,(0,\dots,0,1,0,\dots,0)),
\]
where the “$1$” is placed in the $\alpha$'s component.
Then, for all subsets $U$ of $M$, any (set-theoretic) section $s$ in $E$ defined on $U$ can be expressed uniquely in the form $s = s^\beta e_\beta$, where the $s^\beta$, $\beta = 1, \dots, r$, are functions $s^\beta \colon U \to \R$. For $0 \leq m \leq k$, when $U \subset M$ is open, the section $s$ is of class $C^m$ if and only if $s^\beta$ is of class $C^m$ for all $\beta$.

The \emph{standard connection} on $E$ is defined by the formula
\[
\nabla_U(s) = ds^\beta \otimes (e_\beta|_U),
\]
where $U \subset M$ is open and $s \in C^1(U,E)$. Observe that $\nabla = (\nabla_U)$ is indeed a connection on $E$. As a matter of fact, you obtain $\nabla$ taking $\omega^\alpha_\beta = 0$ for all $\alpha,\beta$ in \cref{r_frame}. Moreover, $\Gamma^\nabla(E)$ is precisely the sheaf of locally constant sections in $E$, a section $s$ in $E$ being called locally constant if its composition with the projection $E \to \R^r$ is a locally constant function.
\end{defi}

To conclude this \namecref{s_connections} we review two important constructions that can be performed with vector bundles and connections---namely, the restriction to an open subspace as well as the more general pullback by a differentiable map.

\begin{rema}[Restriction]
\label{r_restriction}
Let $M$ be a manifold, $E$ a vector bundle over $M$, and $U$ an open subset on $M$. Then we have a natural notion of a restriction of $E$ to $U$, denoted by $E|_U$, which is a vector bundle of class $C^l$ over the open submanifold $M|_U$ of $M$ of class $C^k$. Notice that most often the induced manifold $M|_U$ is sloppily denoted by just $U$.

Assume $1 \leq l$, and let $\nabla$ be a connection on $E$. Then we may sheaf-theoretically restrict $\nabla$ to $U$ to obtain a morphism of abelian sheaves
\[
\nabla|_U \colon C^1(E)|_U \to C^0(T^\vee(M) \otimes E)|_U
\]
on $U$ (or better, on $M|_U$). Since
\begin{align*}
C^1(E)|_U & = C^1(E|_U), \\
C^0(T^\vee(M) \otimes E)|_U & = C^0((T^\vee(M) \otimes E)|_U) = C^0(T^\vee(M)|_U \otimes E|_U) \\
& \cong C^0(T^\vee(M|_U) \otimes E|_U),
\end{align*}
and the Leibniz rule transfers from $\nabla$ down to $\nabla|_U$, we see that $\nabla|_U$ becomes a connection on $E|_U$ (modulo the indicated identification of $T^\vee(M)|_U$ and $T^\vee(M|_U)$ which, in turn, amounts to identifying $T_p(M|_U)$ and $T_p(M)$ for all $p \in U$). This is called the \emph{restriction} of $\nabla$ to $U$. We write $(E,\nabla)|_U$ for the pair $(E|_U,\nabla|_U)$.
\end{rema}

\begin{rema}[Pullback]
\label{r_pullback}
Let $M$ and $M'$ be manifolds of classes $C^k$ and $C^{k'}$ respectively, $1 \leq k,k' \leq \infty$. Let $\phi \colon M' \to M$ be a morphism of class $C^1$ and $E$ a vector bundle of class $C^l$ over $M$, where $1 \leq l \leq k$. Then by a \emph{pullback} of $E$ by $\phi$ I mean a pullback of $E$ by $\phi$ in the sense of $C^1$ manifolds; that is, you first pass from $M$, $M'$, and $E$ to their corresponding manifolds of class $C^1$ (by possibly enlarging the atlases), then you speak of the pullback. Concretely, a pullback of $E$ by $\phi$ is a pair $(E',\phi')$, where $E'$ is a vector bundle of class $C^1$ over $M'$ and $\phi' \colon E' \to E$ is a vector bundle homomorphism covering $\phi$ such that the pair $(E',\phi')$ is universal (or better said, terminal) among all such pairs. We obtain a commutative diagram:
\[
\xymatrix{
E' \ar[r]^{\phi'} \ar[d] & E \ar[d] \\ M' \ar[r]^{\phi} & M
}
\]

Let $(E',\phi')$ be a pullback of $E$ by $\phi$. Moreover, let $\nabla$ be a connection on $E$. I contend there exists a unique connection $\nabla'$ on $E'$ such that the following condition holds: when $U$ is open in $M$, $r \in \N$, $(e_1,\dots,e_r)$ is a frame for $E$ over $U$, and $(\omega^\alpha_\beta)$ an $r\times r$ matrix with values in $C^0(U,T^\vee(M))$ such that
\[
\nabla_U(e_\beta) = \omega^\alpha_\beta \otimes e_\alpha, \quad \forall \beta,
\]
then
\[
\nabla'_{U'}(\phi^*e_\beta) = \phi^*\omega^\alpha_\beta \otimes \phi^*e_\alpha, \quad \forall \beta,
\]
where $U' := \phi^{-1}(U)$.
Note that $\phi^*$ has two meanings: for one, $\phi^*\omega^\alpha_\beta$ denotes the pullback of $\omega^\alpha_\beta$ in the sense of $1$-forms; for another, $\phi^*e_\beta$ denotes the \emph{pullback section} of $e_\beta$ by $\phi$ with respect to the pullback bundle $(E',\phi')$ of $E$; that is, $\phi^*e_\beta$ is the unique section of class $C^1$ in $E'$ defined on $U'$ such that $\phi' \circ \phi^*e_\beta = e_\beta \circ \phi$.

As a matter of fact, for all $U$ and $e = (e_\beta)$ as above, the $r$-tuple $(\phi^*e_\beta)$ constitutes a frame for $E'$ over $U'$---that is, a global frame for $E'|_{U'}$. Therefore, by \cref{r_frame}, there exists a unique connection $\nabla'_{U,e}$ on $E'|_{U'}$ such that
\[
(\nabla'_{U,e})_{U'}(\phi^*e_\beta) = \phi^*\omega^\alpha_\beta \otimes \phi^*e_\alpha, \quad \forall \beta.
\]
We may view $\nabla'_{U,e}$ as a morphism of abelian sheaves
\[
\nabla'_{U,e} \colon C^1(E')|_{U'} \to C^0(T^\vee(M') \otimes E')|_{U'}
\]
on $M'|_{U'}$---note that this is up to the isomorphism $T^\vee(M')|_{U'} \to T^\vee(M'|_{U'})$ which is induced by the inclusion map $U' \to M'$. For another pair $(V,f)$ consisting of an open set $V$ in $M$ and a local frame $f$ for $E$ over $V$ a calculation with transition functions shows that the morphisms of sheaves $\nabla'_{U,e}$ and $\nabla'_{V,f}$ agree on $U' \cap V'$, $V' = \phi^{-1}(V)$. Thus (since the sheaf hom on $M'$ of the abelian sheaves $C^1(E')$ and $C^0(T^\vee(M') \otimes E')$ is again a sheaf on $M'$, and since the $U'$ cover $M'$) there exists a unique morphism of abelian sheaves on $M'$,
\[
\nabla' \colon C^1(E') \to C^0(T^\vee(M') \otimes E'),
\]
such that $\nabla'|_{U'} = \nabla'_{U,e}$ for all pairs $(U,e)$. The Leibniz rule readily extends from the individual $\nabla'_{U,e}$ to $\nabla'$. This proves my claim.

We call $\nabla'$ the \emph{pullback connection} of $\nabla$ associated to the pullback $(E',\phi')$ of $E$ by $\phi$. The pullback connection $\nabla'$ has the following decisive property (slightly generalizing the property used to characterize $\nabla'$ above): for all open subsets $W$ of $M$ and all $\sigma \in C^1(E)(W)$ we have
\[
\nabla'_{\phi^{-1}(W)}(\phi^*\sigma) = \phi^*(\nabla_W(\sigma)),
\]
where the $\phi^*$ on the right-hand side takes a section $\xi$ in $T^\vee(M) \otimes E$ defined on $W$ to the unique section $\phi^*\xi$ in $T^\vee(M') \otimes E'$ defined on $\phi^{-1}(W)$ such that, for all $p' \in \phi^{-1}(W)$, the value $(\phi^*\xi)(p')$ is the image of $\xi(\phi(p'))$ under the evident tensor product map
\[
T_{\phi(p')}^\vee(M) \otimes E_{\phi(p')} \to T_{p'}^\vee(M') \otimes E'_{p'}.
\]
Writing $\xi = \xi^\alpha \otimes e_\alpha$ with respect to a local frame $e = (e_\alpha)$ for $E$, the $\xi^\alpha$ being local sections in $T^\vee(M)$, we have $\phi^*\xi = \phi^*\xi^\alpha \otimes \phi^*e_\alpha$. As a consequence, we see that the pullback of sections
\[
\phi^* \colon C^1(E) \to \phi_*(C^1(E')),
\]
viewed as a morphism of sheaves on $M$, maps the subsheaf $\Gamma^\nabla(E) \subset C^1(E)$ into the subsheaf $\Gamma^{\nabla'}(E') \subset C^1(E')$.
\end{rema}

\section{Basic theory of negligible sets}
\label{s_basics}

I begin by giving the central definition of the text. Further below, I illuminate this definition with several examples and first properties. The idea is to unveil conditions that are necessary for a set $F$ to be negligible in $M$.

\begin{defi}
\label{d_negligible}
Let $M$ be a manifold, $F$ a closed subset of $M$.
\begin{enumerate}
\item Let $(E,\nabla)$ be a vector bundle with connection over $M$. Then $F$ is called \emph{negligible in $M$ for $(E,\nabla)$} when the restriction map
\begin{equation}
\label{e_rest}
\Gamma^\nabla(E)(M) \to \Gamma^\nabla(E)(M \setminus F)
\end{equation}
of the sheaf $\Gamma^\nabla(E)$ is surjective.
\item Let $0 \leq m \leq \infty$. Then $F$ is called \emph{negligible in $M$ for all connections of class $C^m$} when, for all vector bundles $E$ over $M$ and all connections $\nabla$ of class $C^m$ on $E$, the set $F$ is negligible in $M$ for $(E,\nabla)$.
\end{enumerate}
\end{defi}

\begin{rema}[Generalizations]
\label{r_generalizations}
Note that whether a closed set $F \subset M$ is negligible for $(E,\nabla)$ depends exclusively on the presheaf $\Gamma^\nabla(E)$. Following that philosophy, given an arbitrary topological space $M$ and a presheaf (say, of sets) $\mathcal P$ on $M$, one might call a closed subset $F$ of $M$ \emph{negligible in $M$ for $\mathcal P$} when the restriction map $\mathcal P(M) \to \mathcal P(M \setminus F)$ is surjective.

If you prefer to work with open sets instead of closed ones, the latter notion generalizes further. Let $\mathcal C$ be a small category, $T$ a terminal object of $\mathcal C$, and $\mathcal P$ a presheaf of sets on $\mathcal C$. Then call an object $U$ of $\mathcal C$ \emph{full in $\mathcal C$ for $\mathcal P$} (feel free to substitute this expression by an expression that is more to your taste) when the restriction map $\mathcal P(T) \to \mathcal P(U)$ is surjective.
\end{rema}

\begin{exam}[Empty set]
\label{x_emptyset}
Quite trivially, for all manifolds $M$, the empty set is negligible in $M$ for all connections of class $C^0$.
\end{exam}

\begin{exam}[Whole space]
\label{x_wholespace}
For all manifolds $M$, the set $M$ itself is negligible in $M$ for all connections of class $C^0$. Indeed, let $(E,\nabla)$ be a vector bundle with connection over $M$. Then, on the one hand,
\[
\Gamma^\nabla(M \setminus M,E) = \Gamma^\nabla(\emptyset,E) = \{\emptyset\}.
\]
On the other hand, $\Gamma^\nabla(M,E)$ contains the zero section $z$, and $z|_\emptyset = \emptyset$.
\end{exam}

\begin{lemm}
\label{l_sectionsagree}
Let $M$ be a manifold, $(E,\nabla)$ a vector bundle with connection over $M$, $U$ an open, connected subset of $M$, $s,t \in \Gamma^\nabla(U,E)$, $p \in U$ such that $s(p) = t(p)$. Then $s = t$.
\end{lemm}

\begin{proof}[Proof sketch]
Let $q \in U$. Then, due to the connectedness of $U$, there exists a (piecewise) $C^1$ path $\gamma \colon [0,1] \to M$ with $\gamma(0) = p$, $\gamma(1) = q$, and $\gamma([0,1]) \subset U$. Trivializing $E$ locally along $\gamma$, one sees that $s \circ \gamma$ and $t \circ \gamma$ correspond to solutions of an $r$-dimensional system of linear ordinary differential equations defined on $[0,1]$, where $r \in \N$ is the local rank of $E$ along $\gamma$ (compare \cref{r_pullback}). By the uniqueness in the Picard-Lindelöf theorem we deduce the equality $s \circ \gamma = t \circ \gamma$ from
\[
(s \circ \gamma)(0) = s(p) = t(p) = (t \circ \gamma)(0).
\]
Specifically, $s(q) = t(q)$. As $q \in U$ was arbitrary, we are done.
\end{proof}

\begin{prop}
\label{p_local}
Let $M$ be a manifold, $F$ a closed subset of $M$, $(E,\nabla)$ a vector bundle with connection over $M$, $\mathfrak U$ an open cover of $M$ such that
\begin{enumerate}
\item \label{i_localneg} for all $U \in \mathfrak U$, the set $F \cap U$ is negligible in $U$ for $(E,\nabla)|_U$, and
\item \label{i_intersection} for all connected components $W$ of $U \cap V$, where $U,V \in \mathfrak U$, $U \neq V$, there exists an element $p \in W \setminus F$.
\end{enumerate}
Then $F$ is negligible in $M$ for $(E,\nabla)$.
\end{prop}

\begin{proof}
Let $s \in \Gamma^\nabla(M \setminus F,E)$. Then, for all $U \in \mathfrak U$, due to \cref{i_localneg}, there exists an element $\tilde s_U \in \Gamma^\nabla(U,E)$ such that ${\tilde s}_U|_{U \setminus F} = s|_{U \setminus F}$. Therefore there exists a family $(\tilde s_U)_{U \in \mathfrak U}$ such that, for all $U \in \mathfrak U$, $\tilde s_U$ has the aforementioned property. Note that the existence of $(\tilde s_U)$ follows at once invoking the axiom of choice. However, the axiom of choice can be circumvented here. Indeed, for all $U \in \mathfrak U$, there exists a unique extension $\tilde s_U \in \Gamma^\nabla(U,E)$ of $s|_{U \setminus F}$ such that $\tilde s_U = 0$ on all connected components of $U$ that are contained in $F$ (use \cref{l_sectionsagree} for those components of $U$ that are not contained in $F$).

Let $U,V \in \mathfrak U$, $U \neq V$. Let $W$ be a connected component of $U \cap V$. Then by \cref{i_intersection} there exists an element $p \in W \setminus F$. Since both $\tilde s_U$ and $\tilde s_V$ agree with $s$ on $(U \cap V) \setminus F$, we have $\tilde s_U(p) = \tilde s_V(p)$. So $\tilde s_U$ and $\tilde s_V$ agree on $W$ by \cref{l_sectionsagree} as $W$ is connected and open in $M$. As $W$ was arbitrary, we see that $\tilde s_U$ and $\tilde s_V$ agree on $U \cap V$. Thus as $\Gamma^\nabla(E)$ is a sheaf, there exists one, and only one, $\tilde s \in \Gamma^\nabla(E)(M)$ such that $\tilde s|_U = \tilde s_U$ for all $U \in \mathfrak U$. In consequence, for all $U \in \mathfrak U$,
\[
(\tilde s|_{M \setminus F})|_{U \setminus F} = (\tilde s|_U)_{U \setminus F} = \tilde s_U|_{U \setminus F} = s|_{U \setminus F},
\]
which implies that $\tilde s|_{M \setminus F} = s$ as the $U \setminus F$, for $U \in \mathfrak U$, furnish an open cover of $M \setminus F$.
\end{proof}

\begin{prop}
\label{p_components}
Let $M$ be a manifold, $F$ a closed subset of $M$, $(E,\nabla)$ a vector bundle with connection over $M$. Then $F$ is negligible in $M$ for $(E,\nabla)$ if and only if, for all connected components $U$ of $M$, the set $F \cap U$ is negligible in $U$ for $(E,\nabla)|_U$.
\end{prop}

\begin{proof}
The “if” part follows from \cref{p_local} taking $\mathfrak U$ to be the set of connected components of $M$. Note that \cref{i_intersection} of \cref{p_local} holds since any two connected components $U \neq V$ of $M$ have empty intersection, whence $U \cap V$ has itself no connected component at all.

The “only if” part is obtained as follows. Assume that $F$ is negligible in $M$ for $(E,\nabla)$. Let $U$ be a connected component of $M$ and $s \in \Gamma^\nabla(U \setminus F,E)$. Define $t \colon M \setminus F \to E$ to be $s$ on $U \setminus F$ and the zero section outside of $U$. Then $t \in \Gamma^\nabla(M \setminus F,E)$, specifically since $U$ as well as $M \setminus U$ are open in $M$. Thus there exists $\tilde t \in \Gamma^\nabla(M,E)$ such that ${\tilde t}|_{M \setminus F} = t$. In consequence, we have $\tilde s := \tilde t|_U \in \Gamma^\nabla(U,E)$ and
\[
\tilde s|_{U \setminus F} = \tilde t|_{U \setminus F} = t|_{U \setminus F} = s,
\]
which was to be demonstrated.
\end{proof}

\begin{prop}
\label{p_connected}
Let $M$ be a connected manifold, $F$ a closed subset of $M$, $r \in \N$, $1 \leq r$, such that $F$ is negligible in $M$ for $(E,\nabla)$, where $\nabla$ is the standard connection on the trivial bundle $E$ of rank $r$ over $M$. Then $M \setminus F$ is connected.
\end{prop}

\begin{proof}
Since $\Gamma^\nabla(E)$ is the sheaf of locally constant sections in $E$, we see that, for all $U \subset M$ open, $\Gamma^\nabla(E)(U)$ is isomorphic to $(\R^r)^C$ as a real vector space, $C$ denoting the set of connected components of $U$. In particular, $\Gamma^\nabla(E)(M)$ is isomorphic to $(\R^r)^0$ or $\R^r$ (depending on whether or not $M$ contains an element). Thus, considering dimensions, the surjectivity of the restriction map
\[
\Gamma^\nabla(E)(M) \to \Gamma^\nabla(E)(M \setminus F)
\]
implies that the set of connected components of $M \setminus F$ has cardinality $\le 1$. In turn, $M \setminus F$ is connected.
\end{proof}

\begin{exam}[Intervals]
\label{x_interval}
Let $M \subset \R$ be an open interval endowed with its canonical manifold structure (of class $C^k$, $1 \leq k \leq \infty$). Then for a closed subset $F$ of $M$ the following are equivalent:
\begin{enumerate}
\item \label{i_intall} $F$ is negligible in $M$ for all connections of class $C^0$.
\item \label{i_intstd} There exists a number $r \in \N$, $r \geq 1$, such that $F$ is negligible in $M$ for $(E,\nabla)$, where $\nabla$ is the standard connection on the trivial bundle $E$ of rank $r$ over $M$.
\item \label{i_intconcrete} $M \setminus F$ is connected; that is, $F = F \setminus (M \setminus F)$ is the complement in $M$ of an open subinterval of $M$.
\end{enumerate}

\Cref{i_intall} implies \cref{i_intstd} since you can take $r=1$. \Cref{i_intstd} implies \cref{i_intconcrete} by means of \cref{p_connected}. Now, assume \cref{i_intconcrete}. When $F = M$, then $F$ is negligible in $M$ by \cref{x_wholespace}. So, suppose $F \neq M$, so that there exists an element $t_0 \in M \setminus F$. Let $(E,\nabla)$ be a vector bundle with connection over $M$ and $s \in \Gamma^\nabla(M \setminus F,E)$. We know that the bundle $E$ is trivial over $M$ (see \cref{f_vb}). In particular, there exists $r \in \N$ and a global frame $e = (e_1,\dots,e_r)$ for $E$.
Write $s = s^\beta e_\beta$ with functions $s^\beta \colon M \setminus F \to \R$. Then
\[
0 = \nabla_{M \setminus F}(s) = (ds^\alpha + s^\beta\omega^\alpha_\beta) \otimes e_\alpha
\]
for some
\[
\omega^\alpha_\beta = A^\alpha_\beta dx^1 \in C^0(M,T^\vee(M)),
\]
$x^1 \colon M \to \R$ being the identity function on $M$ here.
This amounts to saying that the $r$-tuple of functions $(s^\beta)$ solves the linear differential equation given by the $r\times r$ matrix $-A$ (i.e., the equation $y' = -Ay$), where $A = (A^\alpha_\beta)$.

By the (global) Picard-Lindelöf theorem the linear differential equation given by $-A$ possesses a unique vector solution $(\tilde s^\beta)$ on $M$ such that $\tilde s^\beta(t_0) = s^\beta(t_0)$ for all $\beta$; moreover, ${\tilde s^\beta}|_{M \setminus F} = s^\beta$ since $M \setminus F$ is connected (more so, an interval). Therefore, putting $\tilde s = \tilde s^\beta e_\beta$, we have $\tilde s \in \Gamma^\nabla(E)(M)$ and $\tilde s|_{M \setminus F} = s$. The restriction map in \cref{e_rest} is hence surjective implying that $F$ is negligible in $M$ for $(E,\nabla)$. As $(E,\nabla)$ was arbitrary, we have deduced \cref{i_intall}.
\end{exam}
%
%

\begin{prop}
\label{p_noextension}
For all natural numbers $n \geq 2$ there exists a connection $\nabla$ of class $C^\infty$ on the trivial bundle $E$ of rank $1$ over $\R^n$ such that the following assertions hold:
\begin{enumerate}
\item \label{i_trivialoutside} $\nabla$ is trivial on $\R^n \setminus [-1,1]^n$; that is, $\nabla_{\R^n}(e) = 0$ on $\R^n \setminus [-1,1]^n$, where $e \colon \R^n \to E = \R^n \times \R$ is given by $e(x) = (x,1)$.
\item \label{i_noextension} When $Q = [-1,1]^{n-1} \times [-1,0]$, there exists an element $t \in \Gamma^\nabla(E)(\R^n \setminus Q)$ such that $t$ agrees with $e$ on $\R^n \setminus [-1,1]^n$ and $t$ does not lie in the image of the restriction map
\[
\Gamma^\nabla(E)(\R^n) \to \Gamma^\nabla(E)(\R^n \setminus Q).
\]
\end{enumerate}
\end{prop}

\begin{proof}
Let $n \in \N$, $n \geq 2$. Define $b \colon \R \to \R$ (a “bump”) by
\[
b(x) =
\begin{cases}
e^{-\frac1{1-x^2}} & \text{when $\abs{x} < 1$,} \\
0 & \text{when } \abs x \geq 1.
\end{cases}
\]
Then $b \in C^\infty(\R)$, the support of $b$ is contained in $[-1,1]$, we have $b \geq 0$ everywhere, yet $b$ is not identically zero (since $b(0) = e^{-1} > 0$).
We know there exist $g,h \in C^\infty(\R)$ such that, for all $x \in \R$, we have $g(x),h(x) \geq 0$, 
\[
g(x) = \begin{cases} 0 & \text{when } x < -1, \\ 1 & \text{when } 0 < x,\end{cases} \quad \text{and} \quad h(x) = \begin{cases} 0 & \text{when } 1 < x, \\ 1 & \text{when } x < 0.\end{cases}
\]
As a matter of fact, $g$ can be obtained by substituting $2x + 1$ for $x$ in
\[
\tilde g(x) = \frac{\int_{-1}^x b(y) \,dy}{\int_{-1}^1 b(y) \,dy};
\]
$h$ can then be obtained by letting $h(x) = g(-x)$.

Define
\[
f \colon \R^n \to \R, \quad f(x_1,\dots,x_n) = b(x_1) \cdot \ldots \cdot b(x_{n-1}) h(x_n).
\]
Moreover, for $i = 1, \dots, n$, define $\omega_i \colon \R^n \to \R$ by
\[
\omega_i(x) = - g(x_n) \frac{D_if(x)}{1 + f(x)} \quad \text{for } i \leq n-1, \quad \omega_n(x) = - \frac{D_nf(x)}{1 + f(x)}.
\]
Note that $\omega_i \in C^\infty(\R^n)$ for all $i$. We set
\[
\omega := \omega_1 dx^1 + \dots + \omega_n dx^n \in C^\infty(\R^n,T^\vee(\R^n)).
\]
By \cref{r_frame} there exists a unique connection $\nabla$ on $E$ such that
\[
\nabla_{\R^n}(e) = \omega \otimes e,
\]
where $e \colon \R^n \to E$ is given by $e(x) = (x,1)$; concretely, we have
\[
\nabla_U(s) = \paren{ds_2 + s_2(\omega|_U)} \otimes (e|_U)
\]
for all $U \subset \R^n$ open and all $s \in C^1(U,E)$, where $s_2$ denotes the second component of the function $s \colon U \to E = \R^n \times \R$ (the tensor product of sections is taken pointwise).

I claim that, for all $i = 1, \dots, n$, the function $\omega_i$ has support in $[-1,1]^n$. As a matter of fact, the support of $f$, whence the support of $D_if$, already lies inside $[-1,1]^{n-1} \times (-\infty,1]$. Thus for $i \leq n-1$ it suffices to note that we have $g(x_n) = 0$ for $x_n \in (-\infty,-1)$. For $i = n$ it suffices to note that the support of $h'$ lies inside $[0,1]$ as $h$ is constant on $\R \setminus [0,1]$; thus the support of $\omega_n$ is not only contained in $[-1,1]^n$, but contained in $[-1,1]^{n-1} \times [0,1]$. Anyways, the connection $\nabla$ is trivial on $\R^n \setminus [-1,1]^n$; that is, we have \cref{i_trivialoutside}.

It remains to verify \cref{i_noextension}. For that matter, define
\[
t_2 \colon \R^n \setminus Q \to \R, \quad t_2(x) = 1 + \begin{cases}
0 & \text{when } x_n \leq 0, \\
f(x) & \text{when } 0 < x_n
\end{cases}
\]
and
\[
t \colon \R^n \setminus Q \to E, \quad t(x) = (x,t_2(x)).
\]
Let $x \in \R^n \setminus [-1,1]^n$. Then $t(x) = e(x)$. When $x_n \leq 0$, this is evident. When $0 < x_n$, use the fact that $f$ has support lying inside $[-1,1]^{n-1} \times (-\infty,1]$. In particular, we see that $t$ is $\nabla$-parallel on $\R^n \setminus [-1,1]^n$.

On $\R^{n-1} \times (0,\infty)$, we have
\[
dt_2 + t_2\omega = \sum_{i=1}^n (D_if + (1 + f)\omega_i) dx^i = 0.
\]
In order to see that the summands $i = 1,\dots,n-1$ vanish, note that for $0 < x_n$ one has $g(x_n) = 1$. Thus $t$ is $\nabla$-parallel on $\R^{n-1} \times (0,\infty)$, too, whence $t$ is $\nabla$-parallel on $\R^n \setminus Q$, in symbols: $t \in \Gamma^\nabla(E)(\R^n \setminus Q)$.

Assume that $v \in \Gamma^\nabla(E)(\R^n)$ with $v|_{\R^n \setminus Q} = t$. Let $0' = (0,\dots,0) \in \R^{n-1}$. As noted above, we know that $\omega_n(0',x_n)$ vanishes for $x_n < 0$. Thus by the uniqueness in Picard-Lindelöf's theorem we infer that $v_2(0',x_n) = 1$ for all $x_n < 0$ since $v_2(0',-2) = t_2(0',-2) = 1$. On the other hand, for all $x_n > 0$, we have
\[
v_2(0',x_n) = t_2(0',x_n) = 1 + b(0)^{n-1}h(x_n),
\]
where the right-hand side tends, for $x_n \to 0$, to
\[
1 + b(0)^{n-1}h(0) = 1 + b(0)^{n-1} > 1.
\]
This contradicts the continuity of $v$ (or, more precisely, the continuity of the function $v_2(0',\_) \colon \R \to \R$). Therefore, there does not exist an element of $\Gamma^\nabla(E)(\R^n)$ which yields $t$ when restricted to $\R^n \setminus Q$.
\end{proof}

\begin{coro}
\label{c_nowheredense}
Let $M$ be a connected, Hausdorff manifold of dimension $\geq 2$, $F \neq M$ a closed subset of $M$ such that, for all connections $\nabla$ of class $C^{k-1}$ on the trivial bundle $E$ of rank $1$ over $M$, the set $F$ is negligible in $M$ for $(E,\nabla)$. Then $F$ is nowhere dense in $M$.
\end{coro}

\begin{proof}
Assume that $p$ is an interior point of $F$ in $M$. Then there exists a natural number $n$ and a coordinate chart $x \colon U \to \R^n$ on $M$ at $p$ such that $[-1,1]^n \subset x(U)$ and $K := x^{-1}([-1,1]^n) \subset F$ (in fact, we can achieve $x(U) = \R^n$ and $U \subset F$). As $M$ is connected and $\dim M \geq 2$, we have $n \geq 2$. Since $K$ is compact in $M$ and $M$ is Hausdorff, $K$ is closed in $M$. Thus $\{M \setminus K,U\}$ constitutes an open cover of $M$.

In consequence, there exists a (unique) $1$-form $\tilde\omega \in C^{k-1}(M,T^\vee(M))$ such that $\tilde\omega$ is zero on $M \setminus K$ and
\[
\tilde\omega|_U = x^*(\omega) = (\omega_1 \circ x)dx^1 + \dots + (\omega_n \circ x)dx^n,
\]
where $\omega$ and the $\omega_i$, $i=1,\dots,n$, stem from \cref{p_noextension}---note that these are $C^\infty$ on $\R^n$. By \cref{r_frame} there exists a unique connection $\tilde\nabla$ on $E$ such that
\[
\tilde\nabla_M(e) = \tilde\omega \otimes e,
\]
the section $e \colon M \to E$ being given by $e(q) = (q,1)$. Observe that $\tilde\nabla$ is of class $C^{k-1}$ as $\tilde\omega$ is.
We know there exists $t$ as in \cref{i_noextension} of \cref{p_noextension}. Write $t_2$ for the second component of $t \colon \R^n \setminus Q \to \R^n \times \R$. Then there exists a function $\tilde t \colon M \setminus A \to E$, where $A := x^{-1}([-1,1]^{n-1} \times [-1,0])$, such that
\[
\tilde t(q) = (q,\tilde t_2(q)), \quad \tilde t_2(q) = \begin{cases}
1 & \text{when } q \in M \setminus K, \\
t_2(x(q)) & \text{when } q \in U \setminus A
\end{cases}
\]
for all $q \in M \setminus A$.
Clearly $\tilde t$ is $\tilde\nabla$-parallel since it is parallel on $M \setminus K$ and parallel on $U \setminus A$. Therefore,
\[
\tilde t|_{M \setminus F} \in \Gamma^{\tilde\nabla}(M \setminus F,E).
\]
So as $F$ is negligible in $M$ for $(E,\tilde\nabla)$, there exists a $\tilde\nabla$-parallel section $\tilde v$ on $M$ such that $\tilde v$ restricted to $M \setminus F$ yields ${\tilde t}|_{M \setminus F}$.

Now since $F \neq M$, there exists an element $p' \in M \setminus F$. Thus $\tilde t(p') = \tilde v(p')$. As $M$ is connected, the set $M \setminus A$ is connected in $M$.\footnote{An elementary way to see this is the following: Take $p_1,p_2 \in M \setminus A$. Since $M$ is path-connected, there exists a path $\gamma \colon [0,1] \to M$ from $p_1$ to $p_2$. When $\gamma$ does not meet $A$, we are done. When $\gamma$ meets $A$, let $t_1,t_2$ be the infimum and supremum of $\gamma^{-1}(A)$, respectively. A path $\delta$ from $p_1$ to $p_2$ in $M \setminus A$ is obtained traversing $\gamma$ from $0$ to $t_1 - \epsilon$ (for suitable $\epsilon >0$), then using the chart $x$ to go from $\gamma(t_1 - \epsilon)$ to $\gamma(t_2 + \epsilon)$ inside $U\setminus A$, then traversing $\gamma$ from $t_2 + \epsilon$ to $1$.} So by \cref{l_sectionsagree}, we see that
\[
\tilde v|_{M \setminus A} = \tilde t.
\]
Thus $\tilde v$ furnishes a $\tilde\nabla$-parallel extension of $\tilde t$ to all of $M$. In turn, $\tilde v$ gives rise to a $\nabla$-parallel extension of $t$ to all of $\R^n$. This, however, contradicts \cref{i_noextension} of \cref{p_noextension}.
\end{proof}
%

\section{Extension results}
\label{s_hyper}

In this \namecref{s_hyper} we prove that parallel sections in vector bundles can be extended over certain good-natured sets. The essential technique I use to extend parallel sections is the following result on the regularity of solutions of linear differential equations depending on parameters.

The regularity theorem is formulated for arbitrary Banach spaces, even though in our applications (\cref{l_hyper}) the Banach spaces are finite-dimensional. By a \emph{Banach space} I mean a real Banach space here. Given a Banach space $E$, I write $\cL(E)$ for the Banach space of continuous linear operators on $E$.

\begin{theo}
\label{t_parameters}
Let $J \subset \R$ be an open interval, $t_0 \in J$, $E$ and $F$ Banach spaces, $V \subset F$ open,
\[
A \colon J \times V \to \cL(E)
\]
a continuous map.
\begin{enumerate}
\item \label{i_parameters-ex} There exists a unique continuous map
$
X \colon J \times V \to \cL(E)
$
such that
\[
X(t_0,y) = \id_E, \quad \forall y \in Y,
\]
and
\[
(D_1X)(t,y) = A(t,y)X(t,y), \quad \forall (t,y) \in J \times V.
\]
\item \label{i_parameters-c1} When $X$ is as in \cref{i_parameters-ex} and $A$ is of class $C^1$, then $X$ is of class $C^1$.
\end{enumerate}
\end{theo}

\begin{proof}
When $t_0 = 0$, \cref{i_parameters-ex} is precisely \cite[IV, Proposition 1.9]{lang}. The general statement (arbitrary $t_0$) follows considering the shifted interval $J - t_0$ instead of $J$ and translating $A$ and $X$ accordingly.

Now let $X$ be as described in \cref{i_parameters-ex}. Set $U := V \times \cL(E)$ and define
\[
f \colon J \times U \to F \times \cL(E), \quad f(t,(y,x)) = (0,A(t,y)x).
\]
Regard $f$ as a time-dependent vector field on the open subset $U$ of the Banach space $F \times \cL(E)$. Then, as one verifies easily, the map
\[
\alpha \colon J \times U \to U, \quad \alpha(t,(y,x)) = (y,X(t,y)x)
\]
constitutes a (global) flow for $f$ with initial time $t_0$; that is, for all $(y,x) \in U$ we have
\[
D_1\alpha(t,(y,x)) = f(t,\alpha(t,(y,x))), \quad \forall t \in J, \quad \text{and} \quad \alpha(t_0,(y,x)) = (y,x).
\]
Assume that $A$ is of class $C^1$. Then $f$ is of class $C^1$, mainly because the multiplication on $\cL(E)$, $\cL(E) \times \cL(E) \to \cL(E)$, which is given as the composition of self-maps on $E$, is of class $C^1$. Adapting the proof of \cite[IV, Theorem 1.16]{lang} to cover time-dependent (as opposed to time-independent) vector fields, we find that $\alpha$ is of class $C^1$.\footnote{The decisive point here is to pass from the time-dependent vector field $f$ to its associated time-independent vector field $\bar f$ as explained on \cite[71]{lang}. Then you use the existence of local flows of class $C^1$ for $\bar f$---that is, \cite[IV, Theorem 1.14]{lang}.} In consequence, $X$ is of class $C^1$ as it equals the composition of the embedding $J \times V \to J \times U$, $(t,y) \mapsto (t,(y,\id_E))$, the map $\alpha$, and the projection $F \times \cL(E) \to \cL(E)$ to the second factor. Thus we have proven \cref{i_parameters-c1}.
\end{proof}

\begin{lemm}
\label{l_compactconv}
Let $I,J \subset \R$ be open intervals, $E$ a Banach space, $f \in C^0(I \times J,E)$.
\begin{enumerate}
\item \label{i_compactconv} For all $x_0 \in I$, all $\epsilon > 0$, and all $K \subset J$ compact, there exists a number $\delta > 0$ such that
\begin{equation}
\label{e_compactconv}
\norm{f(x,y) - f(x_0,y)} < \epsilon
\end{equation}
for all $x \in I$ with $\abs{x - x_0} < \delta$ and all $y \in K$.
\item \label{i_diffsecond} When $C$ is nowhere dense in $I$ and $g \in C^0(I \times J,E)$ such that $D_2f = g$ on $(I \setminus C) \times J$ (in particular, $f$ is assumed to be partially differentiable in its second variable on $(I \setminus C) \times J$), then we have $D_2f = g$ everywhere on $I \times J$.
\end{enumerate}
\end{lemm}

\begin{proof}
\Cref{i_compactconv} is fairly standard. The continuity of $f$ in the points of $\{x_0\} \times K$ entails that the set of all sets $U = B_\delta(x_0) \times V$ with $\delta > 0$ and $V \subset \R$ open such that \cref{e_compactconv} holds for all $(x,y) \in U \cap (I \times J)$ furnishes an open cover of $\{x_0\} \times K$. Since $\{x_0\} \times K$ is compact in $\R^2$, there exists a finite subcover. The minimal $\delta$ of the sets $U$ in this subcover then possess the desired property.

Now, let $C$ and $g$ be as in \cref{i_diffsecond}. Then we have
\begin{equation}
\label{e_diffsecond}
f(x,y) - f(x,y_0) = \int_{y_0}^y g(x,\_) \,d\lambda
\end{equation}
for all $x \in I \setminus C$ and all $y_0,y \in J$. Let $x_0 \in I$, $J' \subset J$ a compact subinterval, $\epsilon > 0$. Then by \cref{i_compactconv} there exists a number $\delta > 0$ such that \cref{e_compactconv} holds for all $x \in I$ with $\abs{x - x_0} < \delta$ and all $y \in J'$ as well as for $f$ replaced by $g$ (i.e., you apply \cref{i_compactconv} twice, once for $f$, once for $g$ in place of $f$, then you pass to the minimum of the two $\delta$'s). Since $C$ is nowhere dense in $I$, there exists $x \in I \setminus C$ such that $\abs{x - x_0} < \delta$. Therefore, using \cref{e_diffsecond}, we obtain
\[
\norm*{f(x_0,y) - f(x_0,y_0) - \int_{y_0}^y g(x_0,\_) \,d\lambda} < (2 + \lambda(J'))\epsilon
\]
for all $y_0,y \in J'$. As $\epsilon > 0$ was arbitrary, we deduce
\[
f(x_0,y) - f(x_0,y_0) = \int_{y_0}^y g(x_0,\_) \,d\lambda
\]
for all $y_0,y \in J'$. Indeed, the latter equality holds for all $y_0,y \in J$ since we can pick $J' = [y_0,y]$ or $J' = [y,y_0]$ depending on whether $y_0 \leq y$ or $y < y_0$ (note that these $J'$ are subsets of $J$ as $J$ was assumed to be an interval). In turn, the function $f(x_0,\_) \colon J \to E$ is differentiable on $J$, its derivative being equal to $g(x_0,\_)$. As $x_0 \in I$ was arbitrary, we infer that the function $f$ is partially differentiable on $I \times J$ with respect to the second variable; moreover, $D_2f = g$ holds on $I \times J$, which was to be demonstrated.
\end{proof}
%

\begin{lemm}
\label{l_hyper}
Let $n \in \N$, $2 \leq n$, $M$ an open $n$-dimensional interval---that is, there are open intervals $I_i \subset \R$, $i = 1,\dots,n$, such that $M = I_1 \times \dots \times I_n$ with its canonical manifold structure---, and $F$ a closed subset of $M$.
\begin{enumerate}
\item \label{i_hyperdiscrete} When $C_2 \subset I_2$ is discrete, $b_1 \in I_1$, and
\begin{equation}
\label{e_hyperhalf}
F \subset \{x \in M \mid b_1 \leq x_1, x_2 \in C_2\},
\end{equation}
then $F$ is negligible in $M$ for all connections of class $C^0$.
\item \label{i_hypernowheredense} When $C_2 \subset I_2$ is nowhere dense, $b_1 \in I_1$, and we have \cref{e_hyperhalf}, then $F$ is negligible in $M$ for all connections of class $C^1$.
\item \label{i_hyperbi} When $C_1 \subset I_1$ and $C_2 \subset I_2$ are nowhere dense, $b_1 \in I_1$, $b_2 \in I_2$, and
\[
F \subset \{x \in M \mid b_1 \leq x_1, x_2 \in C_2\} \cap \{x \in M \mid b_2 \leq x_2, x_1 \in C_1\},
\]
then $F$ is negligible in $M$ for all connections of class $C^0$.
\end{enumerate}
\end{lemm}

\begin{proof}
The arguments for the three parts of the lemma all start the same. Let $(E,\nabla)$ be a vector bundle with connection over $M$. Since $M$ is paracompact, Hausdorff, and $C^k$ contractible---that is, there exists a homotopy $M \times [0,1] \to M$ of class $C^k$, in the sense of manifolds with boundary, from the identity on $M$ to a constant map $M \to \{c\}$, $c \in M$---the vector bundle $E$ is trivial as a vector bundle of class $C^l$.\footnote{Note that by \cite[Theorem 1.6]{hatcher} the vector bundle $E$ is trivial in the sense of topological vector bundles. However, the pivotal \cite[Proposition 1.7]{hatcher} carries over nicely to the $C^l$ manifold context; the critical point is to see that in “preliminary facts (1)” of the proof, the patching together of the two trivializations can be realized within class $C^l$, at least for nice $X$.\label{f_vb}} In particular, there exists a global frame $e = (e_1,\dots,e_r)$ for $E$. Denote by
\[
\omega = \omega_1 dx^1 + \dots + \omega_n dx^n \in C^0(M,T^\vee(M))^{r \times r}
\]
the connection form of $\nabla$ with respect to $e$ (see \cref{r_frame}); here, the $\omega_i$ are continuous functions on $M$ with values in $\R^{r \times r}$. We identify $\R^{r\times r}$ with $\cL(\R^r)$.

Let $b_1 \in I_1$ and $C_2 \subset I_2$. Since $I_1$ is open, there exists $a_1 \in I_1$ such that $a_1 < b_1$. By \cref{t_parameters}, \cref{i_parameters-ex}, we know that there exists a unique continuous function
\[
X \colon M = I_1 \times (I_2 \times \dots \times I_n) \to \cL(\R^r)
\]
such that, for all $x' = (x_2,\dots,x_n) \in I_2 \times \dots \times I_n$, we have
\[
X(a_1,x') = \id_{\R^r}
\]
and, for all $x_1 \in I_1$,
\[
(D_1X)(x_1,x') + \omega_1(x_1,x')X(x_1,x') = 0.
\]

Assume \cref{e_hyperhalf} and let $s \in \Gamma^\nabla(M \setminus F,E)$. By abuse of notation, we write $s$ also for the $C^1$ function $M \setminus F \to \R^r$ that represents $s$ with respect to $e$. We define
\[
\tilde s \colon M \to \R^r, \quad \tilde s(x_1,x') = X(x_1,x')s(a_1,x').
\]
Obviously $\tilde s \colon M \to \R^r$ is continuous and partially differentiable in the direction of $x_1$ so that
\[
D_1{\tilde s} + \omega_1{\tilde s} = 0
\]
on $M$---in particular, we see that $D_1{\tilde s}$ is continuous. Since $\nabla_{M \setminus F}(s) = 0$, we have
\[
D_1s + \omega_1s = 0
\]
on $M \setminus F$. Since $F \subset \{x \in M \mid x_2 \in C_2\}$ by \cref{e_hyperhalf}, we have
\[
I_1 \times \{x'\} \subset M \setminus F
\]
for all $x' \in I_2 \times \dots \times I_n$ with $x_2 \notin C_2$. Since $\tilde s(a_1,x') = s(a_1,x')$, the Picard-Lindelöf theorem implies that $\tilde s = s$ on $I_1 \times \{x'\}$ for all such $x'$; that is, $\tilde s$ and $s$ agree on $\{x \in M \mid x_2 \notin C_2\}$.

Suppose that $C_2$ is nowhere dense in $I_2$. Then $I_1 \times C_2 \times I_3 \times \dots \times I_n$ is nowhere dense in $M$. Thus the set where $\tilde s$ and $s$ agree is dense in $M$ and, in turn, dense in $M \setminus F$. Since $\tilde s$ and $s$ are both continuous, this implies ${\tilde s}|_{M \setminus F} = s$. In other words, we have found a continuous extension of $s$ to all of $M$ which already satisfies $\nabla = 0$ in the direction of $x_1$.

Let $i \in \N$, $3 \leq i \leq n$. I claim that $\tilde s$ is partially differentiable in the direction of $x_i$ such that
\[
D_i{\tilde s} + \omega_i{\tilde s} = 0
\]
holds on $M$. For that matter, let $x^0 = (x^0_1,\dots,x^0_n) \in M$. Define
\[
j \colon I_2 \times I_i \to M, \quad j(x_2,x_i) = (x^0_1,x_2,x^0_3,\dots,x^0_{i-1},x_i,x^0_{i+1},\dots,x^0_n)
\]
and
\[
f,g \colon I_2 \times I_i \to \R^r, \quad f = {\tilde s} \circ j, \quad g = - ((\omega_i{\tilde s}) \circ j).
\]
Then $f$ and $g$ are both continuous. Moreover, $f$ is partially differentiable with respect to its second variable on $(I_2 \setminus C_2) \times I_i$ so that
\[
D_2f = (D_i{\tilde s}) \circ j = - ((\omega_i{\tilde s}) \circ j) = g.
\]
Thus by means of \cref{l_compactconv}, \cref{i_diffsecond}, we see that $f$ is partially differentiable with respect to its second variable on $I_2 \times I_i$ such that $D_2f = g$. Specifically, we have this identity in $(x^0_2,x^0_i) \in I_2 \times I_i$. In consequence, $\tilde s$ is partially differentiable in the direction of $x_i$ at $x^0$, and we have
\[
(D_i{\tilde s})(x^0) = (D_2f)(x^0_2,x^0_i) = g(x^0_2,x^0_i) = - (\omega_i\tilde s)(x^0).
\]
As $x^0 \in M$ was arbitrary, my claim is proven.

Under the current assumptions we cannot conclude that $\tilde s$ is partially differentiable in the direction of $x_2$ in points of $F$. We need further suppositions. So, let $x^0 \in F$ and assume that $x^0_2$ is an isolated point of $C_2$. Then an easy application of the mean value theorem shows that $\tilde s$ is partially differentiable in the direction of $x_2$ at $x^0$ with
\[
(D_2{\tilde s})(x^0) = - (\omega_2{\tilde s})(x^0).
\]
This observation proves \cref{i_hyperdiscrete} (all points of $C_2$ are isolated).

When the connection $\nabla$ is of class $C^1$, then the $\omega_i \colon M \to \R^{r \times r} \cong \cL(\R^r)$ are all of class $C^1$ ($i = 1, \dots, n$). The fact that $\omega_1$ is of class $C^1$ implies (by means of \cref{t_parameters}, \cref{i_parameters-c1}) that $X$ and, in turn, $\tilde s$ is of class $C^1$. Specifically, $\tilde s$ is partially differentiable in the direction of $x_2$ with continuous derivative $D_2{\tilde s}$. Since $D_2{\tilde s} + \omega_2\tilde s = 0$ holds on $M \setminus F$ and $M \setminus F$ lies dense in $M$, we infer that the latter equation holds on all of $M$ by the continuity of its left-hand side. This observation proves \cref{i_hypernowheredense}.

In order to prove \cref{i_hyperbi} you conduct the arguments that lead up to \cref{i_hyperdiscrete} once again, only with indices $1$ and $2$ swapped. This procedure yields a second extension $\tilde s_2$ of $s$ to $M$, of which we know a priori that it is partially differentiable in the direction of $x_2$. Since $\tilde s_2$ is, just as $\tilde s$, continuous, we conclude that $\tilde s_2 = \tilde s$ by means of the density of $M \setminus F$ in $M$.
\end{proof}

\begin{coro}
\label{c_counterex-c0}
Let $n$ and $M$ be as in \cref{l_hyper}. Then there exists $F$ such that
\begin{enumerate}
\item \label{i_negc1} $F$ is negligible in $M$ for all connections of class $C^1$---in particular, $F$ is a closed nowhere dense subset of $M$ with $M \setminus F$ being connected---,
\item \label{i_nnegc0} there exists a connection $\nabla$ of class $C^0$ on the trivial bundle $E$ of rank $1$ over $M$ such that $F$ is not negligible in $M$ for $(E,\nabla)$---in particular, $F$ is not negligible in $M$ for all connections of class $C^0$---,
\item \label{i_nullset} $\lambda^n(F) = 0$, where $\lambda^n$ denotes the Lebesgue measure on $\R^n$.
\end{enumerate}
\end{coro}

\begin{proof}
Purely for convenience (i.e., nicer formulas below) let us assume that $0 \in I_1$ and $[0,1] \subset I_2$. Then we can take
\[
F = \{x \in M \mid 0 \leq x_1, x_2 \in C\},
\]
where $C$ denotes the Cantor ternary set. Then $F$ is obviously closed in $M$ and we have \cref{i_nullset} (since $\lambda^1(C) = 0$). Since $C$ is nowhere dense (and closed) in $[0,1]$, it is so in $I_2$. Thus we infer \cref{i_negc1} from \cref{l_hyper}, \cref{i_hypernowheredense}.

In order to see \cref{i_nnegc0}, define
\[
f \colon I_1 \to \R, \quad f(x) = \begin{cases} 0 & \text{when } x < 0, \\ x^2 & \text{when } 0 \leq x, \end{cases}
\]
and let $g \colon I_2 \to \R$ be the extension of the Cantor function such that $g(x) = 0$ for $x < 0$ and $g(x) = 1$ for $1 < x$. Moreover, let $\nabla$ (on $E$, the trivial bundle of rank $1$ over $M$) be given by $\omega_1 dx^1$ with respect to the frame $e = e_1$, where $e(x) = (x,1)$ and
\[
\omega_1(x_1,\dots,x_n) = - \frac{f'(x_1)g(x_2)}{1 + f(x_1)g(x_2)}.
\]
More explicitly,
\[
\nabla_M(se) = (ds + s\omega_1 dx^1) \otimes e
\]
for all $s \in C^1(M)$. Then $\nabla$ is a connection of class $C^0$ on $E$ (\cref{r_frame}). Define
\[
s \colon M \to \R, \quad s(x_1,\dots,x_n) = 1 + f(x_1)g(x_2).
\]
Then $se|_{M \setminus F} \in C^1(M \setminus F,E)$ with $\nabla_{M \setminus F}(se|_{M \setminus F}) = 0$; note that $g$ is differentiable on $I_2 \setminus C$ with $g' = 0$. However, there exists no $\tilde s \in C^1(M)$ such that ${\tilde s}|_{M \setminus F} = s|_{M \setminus F}$ since such a $\tilde s$ would necessarily agree with $s$ (by continuity and the fact that $M \setminus F$ lies dense in $M$), yet the function $s$ is not partially differentiable with respect to its second variable in points $x \in M$ with $0 < x_1$ and $x_2 \in C$ (of which there exists at least one).
\end{proof}

\begin{prop}
\label{p_fatcantor}
Let $n \in \N$, $2 \leq n$, $M$ an open $n$-dimensional interval, and $\lambda_0 < \lambda^n(M) < \infty$. Then there exists $F$ such that $F$ is compact in $M$, $\lambda_0 < \lambda^n(F)$, and $F$ is negligible in $M$ for all connections of class $C^0$.
\end{prop}

\begin{proof}
In case $\lambda^n(M) = 0$ take $F = \emptyset$. Assume $0 < \lambda^n(M) = \lambda(I_1) \cdot \ldots \cdot \lambda(I_n)$ now. Evidently, there exists a number $\delta > 0$ such that
\[
\lambda_0 < (\lambda(I_1) - 2\delta) \cdot \ldots \cdot (\lambda(I_n) - 2\delta) \quad \text{and} \quad 0 < \lambda(I_i) - 2\delta, \quad \forall i.
\]
Moreover, for all $i \in \{1,\dots,n\}$, there exists a compact subinterval $I_i' \subset I_i$ such that $\lambda(I_i) - \delta < \lambda(I_i')$; also, there exists a closed, nowhere dense subset $C_i \subset I_i'$ such that $\lambda(I_i') - \delta < \lambda(C_i)$ (take a “fat” Cantor set, for instance, translated and rescaled to fit inside $I_i'$). Put $F := C_1 \times \dots \times C_n$. Then $F$ is obviously compact in $M$ with
\[
(\lambda(I_1) - 2\delta) \cdot \ldots \cdot (\lambda(I_n) - 2\delta) < \lambda(C_1) \cdot \ldots \cdot \lambda(C_n) = \lambda^n(F).
\]
For all $i$, since $I_i$ is open and $I_i' \subset I_i$ compact, there exists $b_i \in I_i$ such that $I_i' \subset \{x_i \in I_i \mid b_i \leq x_i\}$. Therefore, $F$ is negligible in $M$ for all connections of class $C^0$ by means of \cref{l_hyper}, \cref{i_hyperbi}.
\end{proof}

\begin{coro}
\label{c_bigmeasure}
Let $n \in \N$, $2 \leq n$, $M \subset \R^n$ open (endowed with its canonical manifold structure of class $C^k$, $1 \leq k \leq \infty$).
\begin{enumerate}
\item \label{i_finitemeasure} For all numbers $\lambda_0 < \lambda^n(M)$, there exists a compact subset $F$ of $M$ such that $\lambda_0 < \lambda^n(F)$, yet $F$ is negligible in $M$ for all connections of class $C^0$.
\item \label{i_infinitemeasure} When $\lambda^n(M) = \infty$, there exists a closed subset $F$ of $M$ such that $\lambda^n(F) = \infty$, yet $F$ is negligible in $M$ for all connections of class $C^0$.
\end{enumerate}
\end{coro}

\begin{proof}
We know there exists an at most countable set $\mathfrak Q$ of compact $n$-dimensional cubes of strictly positive measure such that $\bigcup{\mathfrak Q} = M$, the collection $\mathfrak Q$ is locally finite in $M$, and, for all $Q_1, Q_2 \in \mathfrak Q$, $Q_1 \neq Q_2$, the interiors of $Q_1$ and $Q_2$ are disjoint.\footnote{One can take $\mathfrak Q$ to be the set of all $Q \subset M$ such that, for some $k \in \N$, $2^kQ$ is a unit cube with integral vertices and $Q$ is maximal with respect to set inclusion among all of those $Q$.} In particular,
\begin{align*}
\lambda^n(M) = \lambda^n(\bigcup{\mathfrak Q}) \leq \sum_{Q \in \mathfrak Q} \lambda^n(Q) = \sum_{Q \in \mathfrak Q} \lambda^n(Q^\circ) = \lambda^n(\bigcup_{Q \in \mathfrak Q} Q^\circ) \leq \lambda^n(M).
\end{align*}
Let $\lambda_0 < \lambda^n(M)$ be given. Then there exists a finite subset $\mathfrak Q' \subset \mathfrak Q$ such that
\[
\lambda_0 < \sum_{Q \in \mathfrak Q'} \lambda^n(Q^\circ) =: \lambda_1.
\]
Assume $0 \leq \lambda_0$. Then $0 < \lambda_1$ and, for all $Q \in \mathfrak Q'$, we have
\[
\frac{\lambda_0}{\lambda_1} \lambda^n(Q^\circ) < \lambda^n(Q^\circ).
\]
Therefore by \cref{p_fatcantor}, for all $Q \in \mathfrak Q'$, there exists a compact subset $F_Q$ of $Q^\circ$ such that
\[
\frac{\lambda_0}{\lambda_1} \lambda^n(Q^\circ) < \lambda^n(F_Q)
\]
and $F$ is negligible in $Q^\circ$ for all connections of class $C^0$. In turn, there exists a corresponding tuple $(F_Q)_{Q \in \mathfrak Q'}$. Put $F := \bigcup_{Q \in \mathfrak Q'} F_Q$. Then $F$ is compact in $M$ as a finite union of compact subsets of $M$. Moreover,
\[
\lambda_0 = \frac{\lambda_0}{\lambda_1} \sum_{Q \in \mathfrak Q'} \lambda^n(Q^\circ) < \sum_{Q \in \mathfrak Q'} \lambda^n(F_Q) = \lambda^n(\bigcup_{Q \in \mathfrak Q'} F_Q) = \lambda^n(F).
\]
Now let $(E,\nabla)$ be a vector bundle with connection over $M$. Set
\[
\mathfrak U := \{Q^\circ \mid Q \in \mathfrak Q'\} \cup \{M \setminus F\}.
\]
Then $\mathfrak U$ is an open cover of $M$, evidently. Furthermore, \cref{i_localneg,i_intersection} of \cref{p_local} hold (make distinctions as to whether $U$, and possibly $V$, are equal to some $Q^\circ$ or $M \setminus F$). Hence $F$ is negligible in $M$ for $(E,\nabla)$ by \cref{p_local}. As $(E,\nabla)$ was arbitrary, $F$ is negligible in $M$ for all connections of class $C^0$. This proves \cref{i_finitemeasure} in case $0 \leq \lambda_0$; in case $\lambda_0 < 0$ take $F = \emptyset$ (see \cref{x_emptyset}).

Now assume $\lambda^n(M) = \infty$. By \cref{p_fatcantor} there exists a family $(F_Q)_{Q \in \mathfrak Q}$ of compact subsets $F_Q \subset Q^\circ$ satisfying
\[
\frac12 \lambda^n(Q^\circ) < \lambda^n(F_Q)
\]
such that $F_Q$ is negligible in $Q^\circ$ for all connections of class $C^0$.\footnote{Naively one would invoke the axiom of choice in order to conclude here. However, refining the statement of \cref{p_fatcantor}, the axiom of choice can be bypassed. The main point is that in the proof of \cref{p_fatcantor} the choices of $\delta$, $I_i'$, and $C_i$, which lead to $F$, can be made explicit. For instance, $\delta$ may be chosen as the minimum of $\frac13\lambda(I_i)$, $i=1,\dots,n$, and one third of the infimum of $\{t \geq 0 \mid \lambda_0 = (\lambda(I_1) - t) \cdot \ldots \cdot (\lambda(I_n) - t)\}$ in case the latter set is nonempty.} Define $F := \bigcup_{Q \in \mathfrak Q} F_Q$. Then as
\[
\sum_{Q \in \mathfrak Q} \lambda^n(Q^\circ) = \lambda^n(M) = \infty,
\]
we have
\[
\lambda^n(F) = \sum_{Q \in \mathfrak Q} \lambda^n(F_Q) = \infty.
\]
Since $\mathfrak Q$ is locally finite in $M$, the family $(F_Q)$ is locally finite in $M$, too. As the $F_Q$'s are compact whence closed in $M$, their union $F$ is closed in $M$. That $F$ is negligible in $M$ for all connections of class $C^0$ is inferred just like above employing \cref{p_local}. Therefore we have \cref{i_infinitemeasure}.
\end{proof}

Thus far we have established extension results for manifolds $M$ that are open submanifolds of some $\R^n$, $n \in \N$---that is, we have proven local extension results. In principle, we have established extension results only for open $n$-dimensional intervals. To conclude this \namecref{s_hyper}, we prove a global extension result---namely, \cref{t_submfds}. As a preparatory step we notice that negligibility behaves well under diffeomorphisms.

\begin{prop}
\label{p_invariance}
Let $M$ and $M'$ be manifolds of classes $C^k$ and $C^{k'}$ respectively, $1 \leq k,k' \leq \infty$, $\phi \colon M \to M'$ a $C^1$ diffeomorphism, and $F \subset M$ such that $F' := \phi(F)$ is negligible in $M'$ for all connections of class $C^0$. Then $F$ is negligible in $M$ for all connections of class $C^0$.
\end{prop}

\begin{proof}
Let $(E,\nabla)$ be a vector bundle with connection over $M$. As $\phi$ is a $C^1$ diffeomorphism, it has a $C^1$ inverse $\psi \colon M' \to M$. We know there exists a pullback $(E',\psi')$ of $E$ by $\psi$ (see \cref{r_pullback}).\footnote{As a matter of fact, here, the general existence of pullbacks is not needed. When $\pi \colon E \to M$ is the projection of the vector bundle $E$, then take $E'$ to be given by $E$ (as the total space, with its induced $C^1$ structure), $\phi \circ \pi \colon E \to M'$ (as the projection), and the vector space structures that the fibers of $E$ already have; moreover, take $\phi' = \id_E$.} Also, there exists a pullback connection $\nabla'$ of $\nabla$. Let $\sigma$ be a $\nabla$-parallel section in $E$ defined on $M \setminus F$. Then there exists a pullback section $\sigma'$ of $\sigma$, which is a $\nabla'$-parallel section in $E'$ defined on $\psi^{-1}(M \setminus F) = M' \setminus F'$. As $F'$ is negligible in $M'$ for $(E',\nabla')$, there exists a $\nabla'$-parallel section $\tilde\sigma'$ in $E'$ defined on $M'$ such that $\tilde\sigma'|_{M' \setminus F'} = \sigma'$.

As ($\phi$ and whence) $\psi$ is a $C^1$ diffeomorphism, $\psi'$ is a $C^1$ diffeomorphism and, passing from the $C^l$ manifold structure of $E$ to its induced $C^1$ manifold structure, $(E,\psi'^{-1})$ is a pullback of $E'$ by $\psi^{-1} = \phi$. Hence there exists an associated pullback section $\tilde\sigma$ of $\tilde\sigma'$, which is $\nabla$-parallel since $\nabla$ is the pullback connection of $\nabla'$. Moreover, $\tilde\sigma|_{M \setminus F}$ is the pullback section of $\sigma$ with respect to the identity diagram. Thus, $\tilde\sigma|_{M \setminus F} = \sigma$. As $\sigma$ was arbitrary, this proves that $F$ is negligible in $M$ for $(E,\nabla)$. As $(E,\nabla)$ was arbitrary, this proves in turn that $F$ is negligible in $M$ for all connections of class $C^0$.
\end{proof}

\begin{rema}[Maximal extensions]
\label{r_maxext}
Let $M$ be a manifold, $F$ a closed, nowhere dense subset of $M$, $(E,\nabla)$ a vector bundle with connection over $M$, and $s$ an element of  $\Gamma^\nabla(E)(M \setminus F)$. Let $U_i \subset M$ be open and $s_i \in \Gamma^\nabla(E)(U_i)$ such that $s_i|_{U_i \setminus F} = s|_{U_i \setminus F}$, for $i=0,1$. Then $s_0$ and $s_1$ agree on $(U_0 \setminus F) \cap (U_1 \setminus F) = (U_0 \cap U_1) \setminus F$. Since $F$ nowhere dense in $M$, we know that $(U_0 \cap U_1) \cap F$ is nowhere dense in $U_0 \cap U_1$, whence $(U_0 \cap U_1) \setminus F$ is dense in $U_0 \cap U_1$. Thus we have $s_0 = s_1$ on all of $U_0 \cap U_1$ as $s_0$ and $s_1$ are continuous. Note that the latter conclusion holds even though $E$ might be non-Hausdorff since on an open set over which $E$ is trivial, the $s_i$ correspond to continuous maps to the topological space $\R^r$, which is Hausdorff, $r \in \N$ being the local rank of $E$.

As $\Gamma^\nabla(E)$ is a sheaf, this argument shows that there exists one, and only one, $\tilde s \in \Gamma^\nabla(E)(\tilde U)$ such that $\tilde s|_{U_0} = s_0$ holds for all $U_0$ and $s_0$ as above. Note that $M \setminus F \subset \tilde U$ and $\tilde s|_{M \setminus F} = s$ since we can take $U_0 = M \setminus F$ and $s_0 = s$. We call $\tilde s$ the \emph{maximal $\nabla$-parallel extension} of $s$.

The section $\tilde s$ has the property that when $p \in M \setminus \tilde U$, then there does not exist a section $s_0 \in \Gamma^\nabla(E)(U_0)$, where $U_0 \subset M$ is open, $p \in U_0$, and $s_0 = s$ on $U_0 \setminus F$; otherwise we had $\tilde s|_{U_0} = s_0$ implying $U_0 \subset \tilde U$ and thus $p \in \tilde U$, a contradiction.
\end{rema}

\begin{theo}
\label{t_submfds}
Let $M$ be a manifold, $F$ a closed $C^1$ submanifold with boundary of $M$ such that $1 \leq \codim(F,M)$ and such that, for all connected components $F'$ of $F$ with $\codim(F',M) = 1$, we have $F' \cap \partial F \neq \emptyset$. Then $F$ is negligible in $M$ for all connections of class $C^0$.
\end{theo}

\begin{proof}[Idea of the proof]
Say we are given a $\nabla$-parallel section $s \colon M \setminus F \to E$ in a vector bundle $E$. We want to extend $s$ through the points of $F$. Points $p \in F$ where $\codim_p(F,M) \ge 2$ are not a problem (use \cref{i_hyperdiscrete} or \cref{i_hyperbi} of \cref{l_hyper}). Points with $\codim_p(F,M) = 1$ that are boundary points of $F$ (\cref{d_intpoint}) are no problem either (use \cref{i_hyperdiscrete} of \cref{l_hyper}). The problem lies in extending $s$ through codimension $1$ points $p$ of $F$ which are interior points of $F$. The local picture around such points $p$ looks like $M = \R^n$ dissected by $F = \{x \in \R^n \mid x_n = 0\}$. Thus a purely local consideration is out of the question, for locally around $p$ the space $M \setminus F$ is disconnected (and the two “branches” of $s$ might not fit together).

Now instead of trying to prove directly that $s$ extends to $p$, we more or less go in the opposite direction. Assume that $s$ does not extend through $p$. Then I prove that the same is true for all points $q$ in a neighborhood of $p$. Globalizing this, we arrive at an entire connected component $Z$ of interior points of $F$ such that there is no point in $Z$ around which $s$ extends locally. The component $Z$ belongs to a connected component $F'$ of $F$. By our assumptions on $F$, the component $F'$ must contain at least one boundary point. Around this boundary point, however, $s$ extends locally. Thus we find a point in $Z$ around which $s$ extends locally---contradiction.
\end{proof}

\begin{proof}
Let $(E,\nabla)$ be a vector bundle with connection over $M$, $s \in \Gamma^\nabla(E)(M \setminus F)$. By \cref{r_maxext} there is the maximal $\nabla$-parallel extension $\tilde s \colon \tilde U \to E$ of $s$. We want to show that $\tilde U = M$, or, more specifically, that $M \subset \tilde U$.

Assume that $p \in M$. When $p \notin F$, then $p \in \tilde U$ since $M \setminus F \subset \tilde U$. Suppose that $p \in F$ from now on. As $F$ is a $C^1$ submanifold with boundary of $M$ at $p$, there exist $d,c \in \N$, an open neighborhood $U$ of $p$ in $M$, an open subset $V$ of $\R^d \times \R^c$, and a $C^1$ diffeomorphism $\phi \colon U \to V$, where the sets $U$ and $V$ equipped with their induced manifold structures from $M$ and $\R^d \times \R^c$, respectively, such that
\[
\phi(F \cap U) = (H_d \times \{(0,\dots,0)\}) \cap V.
\]
See \cref{e_halfspace} for the definition of $H_d$. By further restricting $\phi$, we can achieve that $V$ is an open hyperrectangle in $\R^d \times \R^c \cong \R^{d + c}$---that is, $V = I_1 \times \dots \times I_n$ with $n = d + c$ for open intervals $I_i \subset \R$, $i=1,\dots,n$---such that $V$ contains $\phi(p)$. Write $\phi(p) = (\phi_1(p),\dots,\phi_n(p))$.
Observe that
\[
1 \leq \codim(F,M) \leq \codim_p(F,M) = c.
\]
So, we have $2 \leq c$ or $c = 1$.

Assume $2 \leq c$. Then $0 = \phi_{n-1}(p) \in I_{n-1}$, so that there exists $b_{n-1} \in I_{n-1}$ with $b_{n-1} < 0$. Put $C_n = \{0\} \cap I_n$ (actually, $C_n = \{0\}$ since $0 = \phi_n(p) \in I_n$). Then $C_n$ is nowhere dense in $I_n$ and
\[
\phi(F \cap U) \subset \{x \in V \mid x_{n-1} = x_n = 0\} \subset \{x \in V \mid b_{n-1} \leq x_{n-1}, x_n \in C_n\}.
\]
Thus $\phi(F \cap U)$ is negligible in $V$ for all connections of class $C^0$ by \cref{l_hyper}, \cref{i_hyperdiscrete}. Therefore by \cref{p_invariance}, $F \cap U$ is negligible in $U$ for all connections of class $C^0$; in particular, $F \cap U$ is negligible in $U$ for $(E,\nabla)|_U$. Hence there exists an element $s_0 \in \Gamma^\nabla(E)(U)$ which agrees with $s$ on $U \setminus F$. So $U \subset \tilde U$ by the definition of $\tilde s$. As a result, we have $p \in \tilde U$.

Now, assume $c = 1$. If $d$ was equal to $0 \in \N$, the point $p$ would be an isolated point of $F$ and, in consequence, $F' = \{p\}$ would be a connected component of $F$ with $\partial F' = \emptyset$, contradictory to our premises. Thus $d \neq 0$; that is, $1 \leq d$.
Assume that $p \in \partial F$. Then $0 = \phi_d(p) \in I_d$, whence $b_d < 0$ for some $b_d \in I_d$. Just like above (for $2\le c$) we derive $p \in \tilde U$ invoking \cref{i_hyperdiscrete} of \cref{l_hyper} and \cref{p_invariance}.

Assume that $p$ is an interior point of $F$. I claim that $F \cap U \subset M \setminus \tilde U$. For that matter, suppose $q \in (F \cap U) \cap \tilde U$. Since there exists a $C^1$ (a $C^\infty$, in fact) diffeomorphism $V \to \R^n$ taking the set $\{x \in \R^n \mid 0 \leq x_d, x_n = 0\}$ to itself, we can assume that $V = \R^n$. Now, since $\phi(U \setminus \tilde U)$ is closed in $V = \R^n$ and contains at least one element (namely $\phi(p)$), there exists an element $x^0 \in \phi(U \setminus \tilde U)$ such that
\[
\epsilon := \abs{x^0 - \phi(q)} = \inf\{\abs{x - \phi(q)} \mid x \in \phi(U \setminus \tilde U)\}.
\]
In consequence, there exists an element $p^0 \in U \setminus \tilde U$ such that $\phi(p^0) = x^0$. Observe that $x^0 \neq \phi(q)$ as otherwise the injectivity of $\phi$ would imply $p^0 = q$ and thus $p^0 \in \tilde U$. So, $0 < \epsilon$. Moreover, observe that $x^0_n = 0$ since $p^0 \in F$. Therefore, there exists an $n$-dimensional Euclidean move (a translation followed by an orthogonal transformation) $\tau \colon \R^n \to \R^n$ taking $\phi(q)$ to the origin, $x^0$ to $(0,\dots,0,\epsilon,0)$, and the set $\{x \in \R^n \mid x_n=0\}$ to itself. Define $J_n := \R$, $J_{n-1} := (0,\infty)$, and, in case $2 < n$, $J_i := (-\delta,\delta)$ with $\delta = \sqrt{\frac1{2(n-2)}} \epsilon$ for $i=1,\dots,n-2$. Moreover, let
\[
V_0 := J_1 \times \dots \times J_n, \quad U_0 := (\tau \circ \phi)^{-1}(V_0).
\]
Then
\[
(\tau \circ \phi)(U_0 \setminus \tilde U) \subset \left\{ y \in V_0 \mid \sqrt{\frac12}\epsilon \leq y_{n-1},y_n = 0 \right\}.
\]
As a matter of fact, when $y \in V_0$ such that $y = \tau(x)$ for an $x \in \phi(U \setminus \tilde U)$, then we have $y_n = 0$ (note that $x_n = 0$ as $U \setminus \tilde U \subset F$) and
\[
\epsilon^2 \leq \abs{x - \phi(q)}^2 = {\abs y}^2 = y_1^2 + \dots + y_{n-1}^2 \leq (n-2)\delta^2 + y_{n-1}^2 = \frac12 \epsilon^2 + y_{n-1}^2.
\]
Accordingly, by \cref{i_hyperdiscrete} of \cref{l_hyper}, $(\tau \circ \phi)(U_0 \setminus \tilde U)$ is negligible in $V_0$ for all connections of class $C^0$. In consequence, by \cref{p_invariance}, $U_0 \setminus \tilde U$ is negligible in $U_0$ for all connections of class $C^0$; note that $(\tau \circ \phi)|_{U_0} \colon U_0 \to V_0$ is a $C^1$ diffeomorphism. Specifically, $U_0 \setminus \tilde U$ is negligible in $U_0$ for $(E,\nabla)|_{U_0}$, whence there exists an element $s_0 \in \Gamma^\nabla(E)(U_0)$ such that $s_0 = \tilde s$ on $U_0 \cap \tilde U$; that is, $s_0 = s$ on $U_0 \setminus F$. The maximality of $\tilde s$ implies $U_0 \subset \tilde U$. As $p^0 \in U_0$, we deduce $p^0 \in \tilde U$---a contradiction. Therefore, for all $q \in F \cap U$, we have $q \notin \tilde U$.

Let $Z$ be the connected component of the interior of $F$ that contains $p$. Then, on the one hand, the arguments of the preceding paragraph, applied to an arbitrary $p' \in Z$ instead of $p$, show that $Z \setminus \tilde U$ is open in $Z$. On the other hand, $Z \setminus \tilde U$ is certainly closed in $Z$ since $\tilde U$ is open in $M$, thus $Z \cap \tilde U$ open in $Z$. As $p \in Z \setminus \tilde U$, we infer $Z = Z \setminus \tilde U$ from the connectedness of $Z$. Let $\bar Z$ be the closure of $Z$ in $F$.
If $Z = \bar Z$, then $Z$ would be a connected component of $F$ with $\codim(Z,M) = 1$ and $\partial Z = \emptyset$, which is impossible under the assumptions of the \namecref{t_submfds}. Thus there exists an element $p'' \in \bar Z$ such that $p'' \notin Z$. Suppose that $p''$ is an interior point of $F$. Then there exists a connected, open neighborhood $U''$ of $p''$ in $F$ such that $U''$ contains only interior points of $F$. As $p''$ lies in the closure of $Z$, the intersection $Z \cap U''$ contains an element. In turn, $Z \cup U''$ is connected. Thus $Z \cup U'' \subset Z$ by the maximality of $Z$. In particular, we conclude $p'' \in Z$---a contradiction. Therefore, $p''$ is not an interior point of $F$, but a boundary point of $F$. Moreover, $\codim_{p''}(F,M) = 1$ as the codimension of $F$ inside $M$ is constant---that is, constantly equal to $1$---on $Z$. Thus we find that $p'' \in \tilde U$ (just as we did for $p$ in place of $p''$ above).
However, we also have $\bar Z \subset F \setminus \tilde U$ because $F \setminus \tilde U$ is closed in $F$ and a superset of $Z$. This is a contradiction.

In conclusion, we see that $p \in F$, or more generally $p \in M$, implies $p \in \tilde U$. Thus $M = \tilde U$ and $\tilde s \in \Gamma^\nabla(E)(M)$ so that $\tilde s|_{M \setminus F} = s$. As $s$ was arbitrary, this tells us that $F$ is negligible in $M$ for $(E,\nabla)$. As $(E,\nabla)$ was arbitrary, we deduce that $F$ is negligible in $M$ for all connections of class $C^0$, which was to be demonstrated.
\end{proof}

\begin{defi}
\label{d_dissected}
We say that a connected manifold $M$ is \emph{dissected by $C^1$ hypersurfaces} when, for all closed, connected $C^1$ submanifolds (meaning without boundary) $F$ of $M$ with $\codim(F,M) = 1$, the space $M \setminus F$ is disconnected (i.e., equal to the disjoint union of two nonempty, open subsets).
\end{defi}

\begin{coro}
\label{c_submfds}
Let $M$ be a connected, Hausdorff manifold which is dissected by $C^1$ hypersurfaces, $2 \leq \dim M$, and $F \neq M$ a closed $C^1$ submanifold with boundary of $M$. Then the following are equivalent:
\begin{enumerate}
\item $F$ is negligible in $M$ for all connections of class $C^0$. \label{i_neg-all}
\item $F$ is negligible in $M$ for all connections of class $C^{k-1}$. \label{i_neg-smooth}
\item $F$ is negligible in $M$ for all connections of class $C^{k-1}$ on the trivial bundle of rank $1$ over $M$. \label{i_neg-special}
\item $F$ is nowhere dense in $M$ and $M \setminus F$ is connected. \label{i_nowdense-conn}
\item $1 \leq \codim(F,M)$ and there exists no connected component $Z$ of $F$ such that $\codim(Z,M) = 1$ and $\partial Z = \emptyset$. \label{i_explicit}
\end{enumerate}
\end{coro}

\begin{proof}
Clearly \cref{i_neg-all} implies \cref{i_neg-smooth}, and \cref{i_neg-smooth} implies \cref{i_neg-special}. \Cref{i_neg-special} implies \cref{i_nowdense-conn} according to \cref{c_nowheredense} (here we use that $M$ is Hausdorff, connected, and of dimension $2$ or greater---observe that for $M = \R$, equipped with its canonical $C^k$ manifold structure, the conclusion fails as shown by \cref{x_interval}) and \cref{p_connected} (observe that the standard connection on the trivial bundle of rank $1$ over $M$ is of class $C^{k-1}$).

Assume \cref{i_nowdense-conn}. Then the codimension of $F$ in $M$ is $\geq 1$ by \cref{r_nowdense-codim} (as a $C^1$ submanifold, $F$ is also a $C^0$ submanifold). Let $Z$ be a connected component of $F$ such that $\codim(Z,M) = 1$ and $\partial Z = \emptyset$. Then, in particular, $Z$ is a closed, connected $C^1$ submanifold of $M$. As $M$ is dissected by $C^1$ hypersurfaces, we infer, on the one hand, that $M \setminus Z$ is disconnected. On the other hand, since $F$ is nowhere dense in $M$, we know that $F \cap (M \setminus Z) = F \setminus Z$ is nowhere dense in $M \setminus Z$. In turn, $M \setminus F = (M \setminus Z) \setminus (F \setminus Z)$ is dense in $M \setminus Z$, whence $M \setminus Z$ is connected as the closure of a connected subspace---contradiction. Therefore, a $Z$ as above cannot exist. Thus we have \cref{i_explicit}.

Finally, from \cref{i_explicit} we obtain \cref{i_neg-all} by means of \cref{t_submfds}.
\end{proof}

\begin{rema}[Manifolds dissected by hypersurfaces]
\label{r_dissected}
Let $M$ be a connected, second-countable, Hausdorff manifold with $H_1(M;\Z/2\Z) \cong 0$ ($H$ denoting singular homology here). I contend that $M$ is dissected by $C^1$ hypersurfaces. In particular, \cref{c_submfds} applies to all such $M$ (assuming $2 \leq \dim M$ in addition, of course). The blatant examples are: $M = \R^n$ or $M = S^n$ for $n \in \N$, $n \geq 2$. Note that the condition $H_1(M;\Z/2\Z) \cong 0$ can be strengthened to $H_1(M;\Z) \cong 0$. One might also require $M$ to be simply connected.

The proof of my assertion consists in a twofold application of the following version of the Poincaré-Lefschetz duality theorem \cite[\nopp VIII, 7.12]{MR1335915}: When $X$ is a second-countable\footnote{I include the hypothesis of second-countability mainly because the definition of Čech cohomology ${\check H}(A,B)$ in \cite[VIII, \S 6]{MR1335915} requires $A$ and $B$ to be locally compact subspaces of some Euclidean neighborhood retract (ENR) $E$. In case $X$ is second-countable, we can take $E = X$ (cf. the addendum to \cite[\nopp VIII, 1.3]{MR1335915}). Nevertheless, \cite[\nopp VIII, 7.10]{MR1335915} adverts to the fact that the second-countability is in fact superfluous.}, Hausdorff topological $m$-manifold, $m \in \N$, $A \subset X$ a closed subspace, then there exists a sequence $(\gamma_i)_{i \in \Z}$ of isomorphisms (of abelian groups or else $R$-modules)
\[
\gamma_i \colon {\check H}^i_c(A;R) \to H_{m-i}(X,X \setminus A;R),
\]
where $R := \Z/2\Z$ (as a group or ring) and ${\check H}_c$ signifies Čech cohomology with compact supports.

As a matter of fact, take $F$ to be a closed, connected $C^1$ submanifold of $M$ such that $\codim(F,M) = 1$. Set $n := \dim M$. Then $1 \leq n$ and $F$ itself is a (second-countable, Hausdorff) topological $(n-1)$-manifold. Thus employing the duality theorem (for $X = F$, $m = n-1$, $A = F$, $i = n-1$)---observe that this is plain Poincaré duality now---, we obtain
\[
{\check H}^{n-1}_c(F) \cong H_0(F) \cong R = \Z/2\Z,
\]
where I suppress the coefficient group (or coefficient ring) $R$ in my notation of (co-)homology.
The last isomorphism is due to the fact that $F$ is nonempty and pathwise connected.
Employing the duality theorem for $X = M$, $m = n$, $A = F$, $i = n-1$ yields
\[
{\check H}^{n-1}_c(F) \cong H_1(M,M \setminus F).
\]
Yet as $\tilde H_1(M) \cong H_1(M) \cong 0$ (by assumption) and $\tilde H_0(M) \cong 0$ (since $M$ is pathwise connected), the long exact sequence in reduced homology associated to the pair $(M,M \setminus F)$---note that $M \setminus F \neq \emptyset$---implies
\[
H_1(M,M \setminus F) \cong {\tilde H}_0(M \setminus F).
\]
Therefore, $M \setminus F$ has precisely two (path-)connected components, whence is disconnected, which was to be demonstrated.
\end{rema}

\printbibliography
\end{document}